\documentclass{amsart}
\usepackage{mathrsfs,amssymb,mathrsfs, multirow,setspace}
\usepackage[a4paper, total={7in, 9in}]{geometry}
\usepackage{enumerate}
\usepackage{tikz}
\usepackage{tikz-cd}
\usepackage{graphicx}
\usepackage{mathtools}

\newtheorem{theorem}{Theorem}[section]
\newtheorem{lemma}[theorem]{Lemma}
\newtheorem{proposition}[theorem]{Proposition}
\newtheorem{corollary}[theorem]{Corollary}

\newtheorem{example}[theorem]{Example}

\newtheorem{remark}[theorem]{Remark}

\input xy
\xyoption{all}

\begin{document}

	\title[The embedding of line graphs associated to the annihilator graph of commutative
rings.]{The embedding of line graphs associated to the annihilator graphs of commutative
rings}
	
	\author[Mohd Shariq, Praveen Mathil, Mohd Nazim, Jitender Kumar]{Mohd Shariq$^{\dagger}$, Praveen Mathil$^{\dagger}$, Mohd Nazim$^{\ddagger}$, Jitender Kumar*$^{\dagger}$}
	% \address{Department of Mathematics, Birla Institute of Technology and Science Pilani, Pilani-333031 (Rajasthan), India}
	% \email{  shariqamu90@gmail.com, maithilpraveen@gmail.com, jitenderarora09@gmail.com}

	%\date{...}
 \begin{abstract}
The annihilator graph $AG(R)$ of the commutative ring $R$ is an undirected graph with vertex set as the set of all non-zero zero divisors of $R$, and two distinct vertices $x$ and $y$ are adjacent if and only if $ann(xy) \neq ann(x) \cup ann(y)$. In this paper, we study the embedding of the line graph of $AG(R)$ into orientable or non-orientable surfaces. We completely characterize all the finite commutative rings such that the line graph of $AG(R)$ is of genus or crosscap at most two. We also obtain the inner vertex number of $L(AG(R))$. Finally, we classify all the finite rings such that the book thickness of $L(AG(R))$ is at most four.
 \end{abstract}

\subjclass[2020]{05C25, 13A70}

\keywords{Local ring, genus and crosscap of a graph, line graph, annihilator graph.\\ *  Corresponding author \\ $^{\dagger}$ Department of Mathematics, Birla Institute of Technology and Science Pilani, Pilani-333031 (Rajasthan), India \\
$^{\ddagger}$School of Basic and Applied Sciences, Faculty of Science and Technology,
JSPM University, Pune-412207, India \\
Email address: shariqamu90@gmail.com, maithilpraveen@gmail.com, mnazim1882@gmail.com, jitenderarora09@gmail.com}

	\maketitle

\section{Introduction}
% Let R be a commutative ring with identity. The annihilator graph of ring $R$ is the undirected graph $AG(R)$ with vertex set  $Z(R)^*=Z(R)\setminus\{0\}$, and two distinct vertices $x$ and $y$ are adjacent if and only if $ann(xy) \neq ann(x) \cup ann(y)$. 
The study of algebraic structures using graph properties has been a significant area of research over the past three decades, leading to many notable results and interesting questions. Numerous papers have been addressed the association of graphs with rings, viz. zero divisor graph \cite{anderson1999zero}, co-maximal ideal graph \cite{ye2012co}, cozero-divisor graph \cite{afkhami2011cozero}, annihilator graph \cite{MR3169557}, intersection graph of ideals \cite{chakrabarty2009intersection}, etc. Badawi \cite{MR3169557} formally defined the annihilator graph of a commutative ring.
 The annihilator graph of ring $R$ is the undirected graph $AG(R)$ with vertex set  $Z(R)^*=Z(R)\setminus\{0\}$, and two distinct vertices $x$ and $y$ are adjacent if and only if $ann(xy) \neq ann(x) \cup ann(y)$.
  Badawi \cite{MR3169557} explored the relationship between the ring-theoretic properties of a commutative ring and the graph-theoretic characteristics of its annihilator graph. It is established that $AG(R)$ is always connected, with a diameter at most $2$, and has a girth at most $4$ provided that $AG(R)$ contains a cycle.
 Afkhami et al. \cite{afkhami2017some} comprehensively characterised all the finite commutative rings $R$ whose annihilator graphs are planar, outer-planar, or ring graphs. Additionally, they characterised all the finite commutative rings $R$ whose annihilator graphs have clique numbers of $1$, $2$, or $3$. Chelvam et al. \cite{chelvam2017genus} determine isomorphism classes of finite commutative rings with identity such that $AG(R)$ has a genus at most one. Further, Selvakumar et al. \cite{Selvakumar} classify the ﬁnite commutative rings such that $AG(R)$ are projective or genus two. Barati et al. \cite{Baratiline}  studied the embedding of the iterated line graphs of $AG(R)$ in a plane,  as well as characterised all the finite commutative rings with respect to their planar indices, ring indices, outer-planar indices, and generalized outer-planar indices. Numerous other researchers have studied and examined the annihilator graphs of commutative rings from various perspectives (see, \cite {badawi2017recent, Nikandish-smd, Nikandish-cmd, Nikandishclring, Nikandish-md}).

In graph theory, the line graph $L(G)$ of the graph $G$ is the graph whose vertex set is all the edges of $G$ and two vertices of $L(G)$ are adjacent if they are incident in $G$. 
% a line graph $L(G)$ is associated with a given graph $G$, where each vertex in $L(G)$ corresponds to an edge in $G$, and two vertices in $L(G)$ are adjacent if and only if the corresponding edges in $G$ share a common vertex.
The structure of any connected graph can be reconstructed from its line graph, establishing a one-to-one correspondence between connected graphs and connected line graphs.  The line graph characterization of graphs associated to algebraic structures, viz.  groups, rings, etc., have been explored in the literature
(see, \cite{bera2022line,parveen2024finite,singh2022graph}).
Barati et al. \cite{barati2021line} completely classify all commutative rings whose zero divisor graphs are line graphs or complement of a line graph. The rings whose cozero-divisor graphs are line graphs have been investigated by Afkhami et al. \cite{afkhami2024line}.

Bénard \cite{MR485482} investigated the genus of line graphs for certain graph classes and established the lower bound for the genus of the line graph of a complete graph. Chiang-Hsieh et al. \cite{MR2735063} examined the genus of line graphs of zero divisor graphs associated with rings and provided a classification of all finite commutative rings with genus not exceeding two. Eric et al. \cite{eric2014some} analysed the girth and clique number of line graphs derived from total graphs of rings and identified the conditions under which these line graphs are Eulerian. Huadong et al. \cite{Huadongcomaximal} investigated the genus of the line graph of co-maximal graph up to 2. 
Motivated by the work related to certain embeddings of the line graphs, in this paper, our aim is to examine the embeddings of line graphs of annihilator graphs on orientable and non-orientable surfaces. 

The structure of the paper is as follows: Section 2 is dedicated to preliminaries. Section 3 classifies all the finite commutative rings for which the graph $L(AG(R))$ is planar and also provides a classification of all the finite commutative rings for which $L(AG(R))$ having genus at most two. Additionally, we precisely identify all the finite commutative rings $R$ such that the crosscap of $L(AG(R))$ is at most two.

 \section{Preliminaries}

Let $G$ be a simple graph with vertex set $V = V (G)$ and edge set $E=E(G)$.
For $v \in V$, the degree of $v$, denoted by $\emph{deg(v)}$, is the number of edges of $G$ which are incident to $v$. A graph $H$  with vertex set $V^{'}= V (H)$ and edge set $E^{'}= E(H)$, is a \emph{subgraph} of $G$ if and only if $V'\subseteq V $ and $E'\subseteq E $.  A \emph{path} in a graph is a sequence of vertices where an edge connects each pair of consecutive vertices. The \emph{path graph} $P_n$ is a path on $n$ vertices.  We say that $G$ is connected if every pair of distinct vertices $u,v \in V (G)$ are joined by a path. A graph $G$ is labelled \emph{complete} if every pair of vertices is joined in $G$. The notation $K_n$ represents the complete graph consisting of $n$ vertices. A \emph{bipartite} graph $G$ is a graph whose vertex set $V (G)$ can be partitioned into two subsets $V_1$ and $V_2$. The edges in such a graph join vertices in $V_1$ to vertices of $V_2$. In particular, if $E(G)$ consists of all possible such edges, then $G$ is referred to as a complete bipartite graph and is denoted by the symbol $K_{m,n}$, where $|V_1| = m$ and $|V_2|=n$.

Throughout this paper, we assume that $R$ is a finite commutative ring with identity, $Z(R)$ its set of zero-divisors, $\text{Nil}(R)$ its set of nilpotent elements, $U(R)$ denote its group of units and $\mathbb{F}_q$ the field with $q$ elements. Every non-unit is a zero-divisor of $R$. For $a \in R$, let $\text{ann}(a) = \{d \in R : da = 0\}$ be the annihilator of $a \in R$. we denote $|{Nil}(R)|^*$ and $|{Z(R)}^*|$ to be the set of non-zero nilpotent elements and non-zero zero divisors of ring $R$.
% In 2014, Badawi \cite{MR3169557} introduced the annihilator graph $\text{AG}(R)$ as the simple graph with vertex set $Z(R)^*$, and two distinct vertices $x$ and $y$ are adjacent if and only if $\text{ann}(xy) \neq \text{ann}(x) \cup \text{ann}(y)$. One can see that the zero-divisor graph $\Gamma(R)$ is a subgraph of the annihilator graph $\text{AG}(R)$.
By the structure theorem of Artinian rings  \cite[Theorem 8.7]{MR3525784}, a finite commutative ring is a direct product of finite local rings.  If $R$ has a unique maximal ideal $\frak{m}$, then $R$ is called a \emph{local ring}. We denote it as $(R,\frak{m})$.
% Throughout this paper, we assume that $R$ is a finite commutative ring with identity, $Z(R)$ its set of zero-divisors, and $\text{Nil}(R)$ its set of nilpotent elements. Let $U(R)$ denote its group of units, $F_q$ the field with $q$ elements, and $R^* = R \setminus \{0\}$. 
The following results are useful for further reference in this paper.
\begin{theorem}
    \cite [Theorem 3.10]{MR3169557}. Let $R$ be a non-reduced commutative ring with $|\text{Nil}(R)^*| \geq 2$ and let ${AG}_N(R)$ be the (induced) subgraph of ${AG}(R)$ with vertices $\text{Nil}(R)^*$. Then $\text{AG}_N(R)$ is complete.
\end{theorem}

% Suppose that $R$ is finite. Then $R$ is isomorphic to the direct product of a finite number of local rings \cite[Theorem 3.1.4]{12}, also see \cite[Theorem 3(4), Chapter 16]{16}. This implies that every finite reduced ring is isomorphic to the direct product of finite fields. If $R$ is also a local ring with a unique maximal ideal $m$, then $m$ is nilpotent and for $x \in R$, we have $x \in m$ if and only if $x$ is a nonunit, if and only if $x$ is nilpotent, if and only if $x$ is a zero divisor. Further, the quotient ring $R/m$ is a finite field and so $|R/m| = p^t$ for some prime $p$ and positive integer $t$. The structure of finite local principal ideal rings is studied in \cite{27}.

\begin{remark}\label{fielddproduct}
    Let $R\cong F_1\times F_2$, where $F_1$ and $F_2$ are finite fields. Then $AG(R)=K_{|F_1|-1,|F_2|-1}$.
\end{remark}

 A graph $\Gamma$ is called \emph{outerplanar} if it can be embedded in the plane such that all vertices lie on the outer face of $\Gamma$. 
In a graph $\Gamma$, the \emph{subdivision} of an edge $(u,v)$ involves removing the edge $(u,v)$ and introducing a new vertex $w$ between $u$ and $v$. This transforms the original edge into two new edges, $(u,w)$ and $(w,v)$. A graph obtained from $\Gamma$ by a sequence of edge subdivisions is called a subdivision of $\Gamma$. Two graphs are \emph{homeomorphic} if both can be obtained from the same graph by subdivisions of edges.
 %In a graph $\Gamma$, the \emph{subdivision} of an edge $(u,v)$ is the deletion of $(u,v)$ from $\Gamma$ and the addition of two edges $(u,w)$ and $(w,v)$ along with a new vertex $w$. A graph obtained from $\Gamma$ by a sequence of edge subdivision is called a subdivision of $\Gamma$. Two graphs are \emph{homeomorphic} if both can be obtained from the same graph by subdivisions of edges.
 A graph $\Gamma$ is \emph{planar} if it can be drawn on a plane without edge crossing. It is well known that every outer-planar graph is a planar graph. The subsequent theorems discuss the planarity and outerplanarity of graphs and their line graphs.

\begin{theorem}\label{outerplanar criteria}\cite{westgraph}
A graph $\Gamma$ is outer-planar if and only if it does not contain a subdivision of $K_4$ or $K_{2,3}$.
\end{theorem}
\begin{theorem}\cite[Theorem 2.1]{MR2784185}
The line graph $L(G)$ of a graph $G$ is outer-planar if and only if $G$ 
 has no subgraph homeomorphic to $K_{2,3},\hspace{.1cm}  K_{1,4}$ or $K_1\vee P_3$. 
\end{theorem}
\begin{theorem}\label{planar criteria}\cite{westgraph}
A graph $\Gamma$ is planar if and only if it does not contain a subdivision of $K_5$ or $K_{3,3}$.
\end{theorem}
\begin{theorem}\cite[Theorem 2]{greenwell1972forbidden}\label{Planar_linegraph}
The line graph $L(G)$ of a graph $G$ is planar if and only if $G$ has no subgraph homeomorphic to $K_{3,3}$,\hspace{.1cm} $K_{1,5}$,\hspace{.1cm} $P_4\vee K_1$,\hspace{.1cm}$K_2\vee\overline{K_3}$.

\end{theorem}

By a surface, we mean a connected two-dimensional real manifold, i.e., a connected topological space such that each point has a neighbourhood homeomorphic to an open disc. It is well-known that any compact surface is either homeomorphic to a sphere, or to a connected sum of $g$ tori, or to a connected sum of $k$ projective planes. We denote by $S_g$ the surface formed by a connected sum of $g$ tori and by $N_k$ a connected sum of $k$ projective planes. The number $g$ is called the genus of the surface $S_g$, and $k$ is called the crosscap of $N_k$. When considering orientability, the surfaces $S_g$ and the sphere are orientable, whereas $N_k$ is not orientable. A simple graph which can be embedded in $S_g$ but not in $S_{g-1}$ is called a graph of genus $g$. Similarly, if it can be embedded in $N_k$ but not in $N_{k-1}$, then we call it a graph of crosscap $k$. We denote $\gamma(G)$ and $\overline{\gamma}(G)$ for the genus and crosscap of a graph $G$, respectively. One easy observation is that $\gamma(H) \leq \gamma(G)$ and $\overline{\gamma}(H) \leq \overline{\gamma}(G)$ for any subgraph $H$ of $G$.

% A compact connected topological space such that each point has a neighbourhood homeomorphic to an open disc is called a surface. For a non-negative integer $g$, let $\mathbb{S}_{g}$ be the orientable surface with $g$ handles. The genus $g(\Gamma)$ of a graph $\Gamma$ is the minimum integer $g$ such that the graph can be embedded in $\mathbb{S}_{g}$, i.e. the graph $\Gamma$ can be drawn into a surface $\mathbb{S}_{g}$ with no edge crossing. Note that the graphs having genus $0$ are planar, and the graphs having genus one are toroidal. The following results are useful in the sequel.

\begin{proposition}{\cite[Ringel and Youngs]{white1985graphs}}
\label{genus}
 Let $m, n$ be positive integers. 
 \begin{enumerate}
     \item[{\rm(i)}]If $n \ge 3$, then $\gamma(K_n) = \left\lceil \frac{(n-3)(n-4)}{12} \right\rceil$.
     \item[{\rm(ii)}]If $m, n\geq 2$, then $\gamma(K_{m,n}) =  \left\lceil \frac{(m-2)(n-2)}{4}  \right\rceil$.
 \end{enumerate}
 \end{proposition}
 
 \begin{lemma}{\cite[Theorem 5.14]{white1985graphs}}
 \label{eulerformulagenus}
 Let $\Gamma$ be a connected graph with a 2-cell embedding in $\mathbb{S}_{g}$. Then $v - e + f = 2 - 2g$, where $v, e$, and $f$ are the number of vertices, edges, and faces embedded in $\mathbb{S}_{g}$, respectively, and $g$ is the genus of the graph $\Gamma$. 
\end{lemma}
% \begin{lemma}\cite{white2001graphs}
% \label{genusofblocks}
% The genus of a connected graph $\Gamma$ is the sum of the genera of its blocks.
% \end{lemma}

\begin{proposition}\label{crosscap}{\cite[Ringel and Youngs]{mohar2001graphs}}
 Let $m, n$ be positive integers. Then\\
$ {\rm(i)}~~ \overline{\gamma}(K_n) =
 \begin{cases} 
      \left\lceil \frac{(n-3)(n-4)}{6} \right\rceil&  \textit{if}~~ n\geq 3 \\
    3 & \textit{if}~~ n =7 
   \end{cases}$\\
${\rm(ii)}~~ \overline{\gamma}(K_{m,n}) =  \left\lceil \frac{(m-2)(n-2)}{2}  \right\rceil~~ \textit{if}~~ m, n\geq 2$.
\end{proposition}
% \begin{proposition}\cite{MR2735063}
%     \begin{enumerate} \item  $g(L(G)) \leq g(G) + \sum_{n\in(V(G))}g(K_{d(v)+1})$;
% \item $\overline{ \gamma}(L(G)) \leq \overline{ \gamma}(G) +\sum_{n\in(V(G))}\overline{\gamma}(K_{d(v)+1})$
% \end{enumerate}
% \end{proposition}
 Bénard \cite{MR485482}, Chiang-Hsieh and Lee et al. \cite{MR2735063} obtained the formulae of genus and crosscap of the line graph of a complete graph or a complete bipartite graph. See the following lemma.
\begin{lemma}{\cite[Lemma 2.7]{MR2735063}}\label{linecompltegraph}
Let $n$ be a positive integer. Then
\begin{itemize}
    \item [(i)]  $\gamma(L(K_n)) \geq \frac{1}{12}(n + 1)(n-3)(n-4)$;
  \item [(ii)] $\overline{\gamma} (L(K_n)) \geq \frac{1}{6}(n + 1)(n-3)(n-4)$.
\end{itemize}
\end{lemma}
\begin{lemma}{\cite[Lemma 2.8]{MR2735063}} For a positive integer $n$, following holds:
\begin{itemize}
    \item[(i)] $\gamma(L(K_{1,n})) = \frac{1}{12} 
 (n-3)(n-4).$
 \item [(ii)] $\overline{\gamma}(L(K_{1,n})) = \frac{1}{6}
 (n-3)(n-4), if n \neq 1,7;$ and $\overline{\gamma}(L(K_{1,7})) =3$
\end{itemize}
\end{lemma}

 \begin{lemma}{\cite[Lemma 2.9]{MR2735063}}\label{genusK_{2,n}}
 For a positive integer $n$, following holds:
 \begin{itemize}
     \item[(i)] $\gamma(L(K_{2,n})) \geq \frac{1}{6} 
 (n-2)(n-3).$ The equality holds if
 $n \not\equiv 5$ or $9 \pmod{12}$.
  \item [(ii)] $\overline{\gamma}(L(K_{2,n})) \geq \frac{1}{3}
  (n-2)(n-3)$. The equality holds if $n\neq 6$ and $n \not\equiv 1$ or $4 \pmod{6}$.
\end{itemize}
 \end{lemma}

\begin{lemma}{\cite[Lemma 2.13]{MR2735063}}\label{genusoftwograph}: Let $\Gamma$ be a connected graph and $\Gamma_1$, $\Gamma_2$  be two connected subgraphs of $\Gamma$. Suppose the following hold:
\begin{itemize}

  \item[(i)] $ V (\Gamma_1) \cap V (\Gamma_2) = \phi$.
\item[(ii)]  $\Gamma_1$ is a 3-connected planar graph.
\item[(iii)]  Every vertex of $\Gamma_1$ is joined with $\Gamma_2$.
\end{itemize} Then $\gamma(\Gamma) > \gamma(\Gamma_2)$ and $\overline{\gamma}(\Gamma) >\overline{\gamma}(\Gamma_2)$.
\
\end{lemma}

\begin{lemma}\label{eulerformulacrosscap}{\cite[Lemma 3.1.4]{mohar2001graphs}}
Let $\phi : \Gamma \rightarrow \mathbb{N}_{k}$ be a $2$-cell embedding of a connected graph $\Gamma$ to the non-orientable surface $\mathbb{N}_{k}$. Then $v - e + f = 2 - k$, where $v, e$ and $f$ are the number of vertices, edges and faces of $\phi(\Gamma)$ respectively, and $k$ is the crosscap of $\mathbb{N}_{k}$.
\end{lemma}

% \begin{definition}\cite{white2001graphs}
%  A graph $\Gamma$ is orientably simple if $\mu(\Gamma) \neq 2-\overline{\gamma}(\Gamma)$, where $\mu(\Gamma) = \rm{max} \{2- 2$$g$$ (\Gamma), 2-\overline{\gamma}(\Gamma)\}$.

% \end{definition}

% \begin{lemma}\cite{white2001graphs}\label{crosscapofblocks}
% Let $\Gamma$ be a graph with blocks $\Gamma_1, \Gamma_2, \cdots, \Gamma_k$. Then 
% \begin{center}
% $\overline{\gamma}(\Gamma) =$
% $\begin{cases} 
% 1-k+ \sum \limits_{i=1}^{k} \overline{\gamma}(\Gamma_i),  & ~~\textit{if}~~ \Gamma~~ \textit{is orientably simple}\\
% 2k - \sum \limits_{i=1}^{k} \mu (\Gamma_i), & ~~\textit{otherwise.}
% \end{cases}$
% \end{center}
% \end{lemma}

We use the following results frequently in this paper.
% \begin{remark}\label{triangularface}
% For a simple graph, $\Gamma$, we have $2e \geq 3f$. 
% \end{remark}

\begin{lemma} {\cite[Lemma 4.1]{MR2735063}}\label{genusintermofdegree}  Let $\Gamma$ be a simple graph and $u, v$ be two distinct vertices of $\Gamma$ such that $deg(u) = m$ and $deg(v) = n$.
Then $\gamma(L(\Gamma))\geq {\gamma(K_m) +\gamma(K_n)}$;
and $\overline{\gamma}(L(\Gamma))\geq \overline{\gamma}(K_m) +\overline{\gamma} (K_n) + \delta$,
where $\delta= 0$ if $(m - 7)(n - 7) \neq 0$ and $\delta = -1$,
otherwise.
\end{lemma}
\begin{corollary}{\cite[Corollary 4.2]{MR2735063}}\label{genusbydegree}
     Let $\Gamma$ be a simple graph. If $u, v, w$ are three distinct vertices of $\Gamma$ such that $ deg(u) \geq 5$, $deg(v) \geq 5$, and $deg(w) \geq 7$, then $\gamma(L(\Gamma)) \geq 3$ and
$\overline{\gamma} (L(\Gamma)) \geq 3$.
\end{corollary}

%%%%%%%%%%%%%%%%%%%%%%%%%%%%%%%%%%%%%%%%%%%%%%%%%%%%%%%%%%%%%%%%%%%%%%%%%%%%%%%%%%%%%%%%%%%%%%%%%%%%%%%%%%%%%%%%%%%%%%%%%%%%%%%%%%%%%%%%%%%%%%%%%%%%%% 
% \newpage 

\section{Embedding of $\text{L(AG}(R))$ on surfaces}
In this section, we study the embedding of the line graph of annihilator graph ${L(AG}(R))$ on a surface without edge crossing. We begin with the investigation of an embedding of ${L(AG}(R))$ on a plane.

\subsection{Planarity of $\text{L(AG}(R))$}
In this section, we will classify all finite commutative rings with unity, in which ${L(AG}(R))$ is planar and outerplanar, respectively.
\begin{theorem}\cite[Theorem 2]{greenwell1972forbidden}\label{Planar_linegraph}
The line graph $L(G)$ of a graph $G$ is planar if and only if $G$ has no subgraph homeomorphic to $K_{3,3}$,\hspace{.1cm}$K_{1,5}$,\hspace{.1cm}$P_4\vee K_1$,\hspace{.1cm}$K_2\vee\overline{K_3}$.

\end{theorem}
    
\begin{theorem}\label{Planar_line annihilator(non-local)}
 Let $R$ be a non-local commutative ring. Then the graph ${L(AG(R))}$ is planar if and only if ${R}$ is isomorphic to one of the following rings:
\[{\mathbb{Z}_2\times \mathbb{Z}_2},\hspace{.2cm} {\mathbb{Z}_2\times \mathbb{Z}_3},\hspace{.2cm}  {\mathbb{Z}_2\times \mathbb{F}_4}, \hspace{.2cm} {\mathbb{Z}_2\times \mathbb{Z}_5},\hspace{.2cm} {\mathbb{Z}_3\times \mathbb{Z}_3},\hspace{.2cm} {\mathbb{Z}_3\times \mathbb{F}_4},\hspace{.2cm} {\mathbb{Z}_2\times \mathbb{Z}_4},\hspace{.2cm}  {\frac{\mathbb Z_2[x]}{(x^2)}\times {\mathbb {Z}_2}}, \hspace{.2cm} {\mathbb{Z}_2\times \mathbb{Z}_2\times \mathbb{Z}_2}.\]
\end{theorem}
\begin{proof}
Let $R$ be a non-local ring. Then $R \cong R_1 \times R_2 \times \cdots \times R_n$, where each $R_i$ is a local ring with maximal ideal $\frak{m_i}$.
Assume that $L(AG(R))$ is planar. Let $n = 4$. Then the subgraph  induced by the set $S$=$\{(1,0,0,0)$, $(1,1,0,0)$, $(1,1,1,0)$, $(0,0,0,1)$, $(0,0,1,1)$, $(0,1,1,1)\}$ is isomorphic to $K_{3,3}$.  Therefore, $L(AG(R))$ is not planar for $n = 4$ by Theorem \ref{Planar_linegraph}. Consequently, for $n\geq 4$,  $L(AG(R))$ is not planar. Therefore, $n \leq 3$.
%Consequently, $R\cong R_1 \times R_2 \times R_3$, where each $R_i$ is local ring.
% $R\cong R_1 \times R_2 \times R_3$, where each $R_i$ is a field. 

We may now suppose that $R \cong R_1 \times R_2 \times R_3$. First, assume that at least one of $R_i$ is not a field. Without loss of generality, suppose $R_1$ is not a field. Then $|\frak{m_1}| \geq 2$ and $|U(R_1)| \geq 2$. Consider $x_1=(1,0,0)$, $x_2=(u,0,0)$, $x_3=(1,1,0)$, $x_4=(a,0,1)$ and $x_5=(b,1,1)$, where $a,b \in \frak{m_1}\setminus\{0\}$ such that $ab=0$ and $u \in U(R_1)\setminus\{1\}$. 
Note that, for $1\leq i\leq3,$ $x_5 \in \text{ann}(x_4x_i)$ but $x_5 \notin \text{ann}(x_4) \cup \text{ann}(x_i)$, $x_4 \in \text{ann}(x_5x_i)$ but $x_4 \notin \text{ann}(x_5) \cup \text{ann}(x_i)$. Therefore, the subgraph  induced by the set $S^{'}=\{x_1,x_2,x_3,x_4,x_5\}$ is isomorphic to $K_2 \vee\overline{K_3}$, a contradiction. Thus, $R_i$  is a field for each $i\in\{1,2,3\}$. Now, suppose $|R_i| \geq 3$ for some $i$. Without loss of generality, assume that $|R_1| \geq 3$. Consider the set $T=\{(0,0,1),(1,0,1),(u,0,1),(0,1,0),(1,1,0),(u,1,0)\}$. Note that the subgraph induced by $T$ is isomorphic to $K_{3,3}$. It implies that $|R_i| \leq 2$ for each $i$. Thus, $R$ is isomorphic to $\mathbb{Z}_2 \times \mathbb{Z}_2 \times \mathbb{Z}_2$. By Figure 2, the graph $L(AG(R))$ is planar.

Now, we suppose that $R \cong R_1 \times R_2 $. If both $R_1$ and $R_2$ are not fields, then $|\frak{m_i}| \geq 2$ and $|U(R_i)| \geq 2$ for all $i$.  Consider the set $S=\{(a,0), (1,0), (u,0), (0,b), (0,1), (0,v)\}$, where  $u \in U(R_1)\setminus\{1\}$, $v \in U(R_2)\setminus\{1\}$, $a \in \frak{m_1}\setminus\{0\}$, and $b \in \frak{m_2}\setminus\{0\}$. Observe that the subgraph induced by the set $S$ is isomorphic to $K_{3,3}$, leading to a contradiction. Thus, at least one of the rings is a field. We assume that exactly one of $R_i$ is a field. Without loss of generality, assume that $R_2$ is a field.

If $|\frak{m_1}|\geq3$, then $|U(R_1)|\geq4$. We can always choose  $a, b\in \frak{m_1}$ such that $ab=0$. Consider vertices $x_1=(a,1), x_2=(b,1), x_3=(1,0), x_4=(u_1,0),$ and $x_5=(u_2,0)$, where $u_1, u_2\in U(R_1)\setminus\{1\}$. Note that $x_3\in \text{ann}(x_1x_2)$, but $x_3\notin { \text{ann}(x_1)\cup \text{ann}(x_2) }$, $x_2\in \text{ann}(x_1x_i)$ but $x_2\notin ({ \text{ann}(x_1)\cup \text{ann}(x_i)})$, $x_1\in \text{ann}(x_2x_i)$ but $x_1\notin ({ \text{ann}(x_2)\cup \text{ann}(x_i)})$, for $i\in\{3,4,5\} $. Note that, the subgraph induced by the set $P=\{x_1, x_2, x_3, x_4, x_5\}$ is isomorphic to $K_2\vee\overline{K_3}$. It implies that $L(AG(R))$ is not planar. Thus,  $|\frak{m_1}|=2$. Therefore, $R_1$ is either $\mathbb{Z}_4$ or  $\frac{\mathbb Z_2[x]}{(x^2)}$.
Let $|R_2|\geq3$. Now , consider the vertices $x_1=(a,0)$, $x_2=(1,0)$, $x_3=(u,0)$, $x_4=(0,1)$, $x_5=(0,v)$, $x_6=(a,1)$, where $a\in \frak{m_1}$, $u\in U(R_1)\setminus\{1\}$ and $v\in U(R_2)\setminus\{1\}$. Note that $x_1x_4=0=x_1x_5=x_1x_6=x_2x_4=x_2x_5=x_3x_4=x_3x_5$ and $x_6\in ann(x_ix_6)$ but $x_6\notin {ann(x_i)\bigcup ann(x_6)}$, for $i\in\{2,3\}$. Therefore, the subgraph induced by the set $Q=\{x_1,x_2,x_3,x_4,x_5,x_6\}$ is isomorphic to $K_{3,3}$. Consequently, the graph
$L(AG(R))$ is not  planar. It implies that $|R_2|\leq2$. Thus, $R$ is isomorphic to either ${\mathbb{Z}_4\times \mathbb{Z}_2}$ or ${\frac{\mathbb Z_2[x]}{(x^2)} \times {\mathbb {Z}_2}}$.

 Now suppose that both  $R_1$ and $R_2$ are fields. Let $|R_i|>5$, for some $i$. Without loss of generality, assume that  $|R_2|>5$. Note that subgraph  induced by the set $ W=\{(1,0), (0,1),(0,u_1),(0,u_2),(0,u_3),(0,u_4)\}$ is isomorphic to $K_{1,5}$, where $u_i\in U(R_2)\setminus\{1\}$  and so $L(AG(R))$ is not planar. It implies that $|R_i|\leq5$ for each $i$. Suppose that $|R_i|\geq4$ for both $i$. Since $AG(R_1 \times R_2)= K_{|R_1|-1,|R_2|-1}$. Observe that the graph $AG(R_1 \times R_2)$ has $K_{3,3}$ as an induced subgraph. Thus, the graph $L(AG(R))$ is not planar. Hence, $R$ is isomorphic to one of the rings: ${\mathbb{Z}_2\times \mathbb{Z}_2}$, ${\mathbb{Z}_2\times \mathbb{Z}_3}$, ${\mathbb{Z}_2\times \mathbb{F}_4}$, ${\mathbb{Z}_2\times \mathbb{Z}_5}$, ${\mathbb{Z}_3\times \mathbb{Z}_3}$, ${\mathbb{Z}_3\times \mathbb{F}_4}$, ${\mathbb{Z}_3\times \mathbb{Z}_5}$.  If $R={\mathbb{Z}_3\times \mathbb{Z}_5}$, then the subgraph induced by the set $S=\{(1,0), (2,0), (0,1), (0,2), (0,3) (0,4)\}$, is homeomorphic to  $K_2\vee\overline{K}_3$. Therefore, the graph $L(AG(R))$ is not planar. Conversely, if $R={\mathbb{Z}_3\times {\mathbb F_4}}$, then ${\text{AG}({\mathbb{Z}_3\times \mathbb{F}_4})}=K_{2,3}$. By Figure 1,
${{L(AG}({\mathbb{Z}_3\times \mathbb{F}_4})}$ is planar. If $R={\mathbb{Z}_3\times {\mathbb Z_3}}$, then ${{L(AG}({\mathbb{Z}_3\times \mathbb{Z}_3}))}=C_4$, and so  ${{L(AG}({\mathbb{Z}_3\times \mathbb{Z}_3}))}$ is planar. Suppose $R={\mathbb{Z}_4\times {\mathbb Z_2}}$ or ${\frac{\mathbb Z_2[x]}{(x^2)}\times {\mathbb {Z}_2}}$. Note that $AG( {\frac{\mathbb Z_2[x]}{(x^2)}\times {\mathbb {Z}_2}})=AG({\mathbb{Z}_4\times \mathbb{Z}_2})=K_{2,3}$, and so $L(AG( {\mathbb{Z}_4\times {\mathbb {Z}_2}}))=L(K_{2,3})$. By the Figure 1. $L(AG(\mathbb{Z}_4\times {\mathbb {Z}_2}))$ is planar. If $R={\mathbb{Z}_2\times {\mathbb Z_3}}$ or ${\mathbb{Z}_2\times {\mathbb Z_2}}$, then ${{L(AG(R))}}$ is $K_{1,2}$ or $K_2$, respectively and so the graph ${{L(AG(R))}}$ is planar.

\begin{figure}[h!]
\centering
\includegraphics[width=0.6 \textwidth]{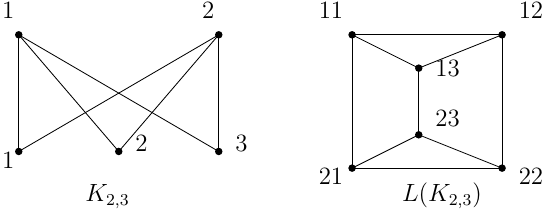}
\caption{Line graph of $K_{2,3}$}
\label{line k}
\end{figure}

\begin{figure}[h!]
\centering
\includegraphics[width=0.6 \textwidth]{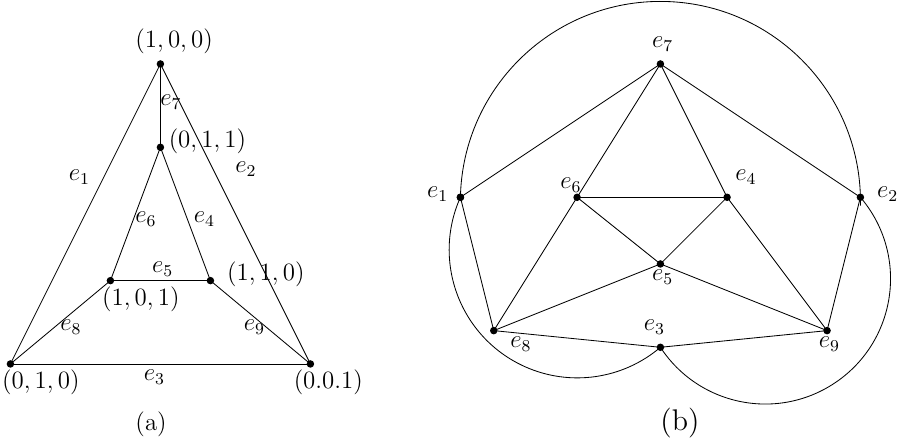}
\caption{ $\textbf{
(a)}$ $AG(\mathbb{Z}_2\times\mathbb{Z}_2\times\mathbb{Z}_2)$ $\hspace{0.2cm}$  $\textbf{(b)}$ Line graph of $AG(\mathbb{Z}_2\times\mathbb{Z}_2\times\mathbb{Z}_2)$}
\label{line k}
\end{figure}
\end{proof}
\begin{theorem}\label{localringsplnr} For the finite local ring $(R,\frak{m})$, the line graph $L(AG(R))$ is planar if and only if the annihilator graph $AG(R)$ is planar.
\begin{proof}
    Suppose $AG(R)$ is planar. Since $AG(R)$ is complete, we obtain that $AG(R)=K_r$, where $r=|\frak{m}^*|$ and $r\leq4$. It implies that $AG(R)$ has no subgraph isomorphic to  $K_{3,3}$,\hspace{.1cm} $K_{1,5}$,\hspace{.1cm} $P_4\vee K_1$,\hspace{.1cm}$K_2\vee\overline{K_3}$. Therefore,  $L(AG(R))$ is planar.
    Conversely, suppose  $L(AG(R))$ is planar. If $r\geq5$, then  $AG(R)$ has $K_5$. It follows that $AG(R)$ contains $P_4\vee K_1$ as a subgraph. It implies that $L(AG(R))$ is non-planar, a contradiction. Therefore, $r\leq 4$. Hence,  $AG(R)$ is planar.
\end{proof}

\end{theorem}

In the view of Theorem \ref{localringsplnr} and {\cite[Theorem 14]{chelvam2017genus}}, we have the following corollary:

\begin{corollary}
    
\label{planar_lineannhilator(local)} For the finite local ring $R$, the line graph $L(AG(R))$ is planar if and only if $R$ is isomorphic to one of the following rings:
\[\mathbb{Z}_4, \hspace{.2cm}{\frac {\mathbb {Z}_{2}[x]}{(x^2)}},\hspace{.2cm} \mathbb{Z}_9,\hspace{.2cm}\frac{\mathbb {Z}_{3}[x]} {(x^3)},\hspace{.2cm}\mathbb{Z}_8,\hspace{.2cm} {\frac {\mathbb {Z}_{2}[x]}{(x^2)}},  \frac{\mathbb  {Z}_4[x]}{(2x,x^2-2)},\hspace{.2cm} \frac{ \mathbb {F}_{4}[x]}{(x^2)},\hspace{.2cm} \frac{\mathbb {Z}_{4}[x]}{(x^2+x+1)},\hspace{.2cm}  \frac{\mathbb {Z}_{4}[x]}{(x,2)},\hspace{.2cm}  \frac{\mathbb {Z}_{2}{[x,y]}}{(x,y)^2}, \hspace{.2cm} \mathbb Z_{25},\hspace{.2cm} \frac{\mathbb {Z}_{5}[x]}{(x^2)}.\]
    \end{corollary}

\begin{figure}[h!]
\centering
\includegraphics[width=0.6 \textwidth]{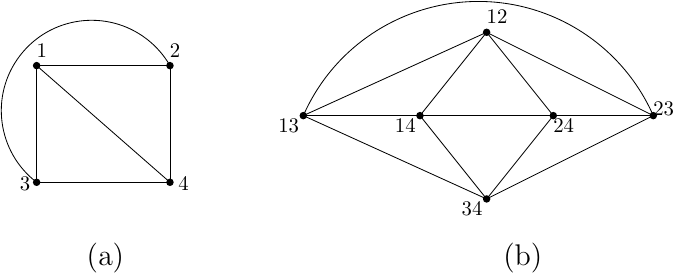}
\caption{ $\textbf{
(a)}$ $AG(\mathbb{Z}_{25})$ $\hspace{0.2cm}$  $\textbf{(b)}$ Line graph of $AG(\mathbb{Z}_{25}$)}
\label{line k}
\end{figure}

\begin{theorem}\cite[Theorem 2.1]{MR2784185}
The line graph $L(G)$ of a graph $G$ is outer planar if and only if $G$ 
 has no subgraph homeomorphic to $K_{2,3}, K_{1,4}$ or $K_1\vee P_3$. 
\end{theorem}

\begin{theorem}\label{outerPlanar_line annihilator(non-local)}
 Let $R$ be a non-local commutative ring. Then the graph $L(AG(R))$ is outer planar if and only if ${R}$ is isomorphic to one of the four rings:
\[{\mathbb{Z}_2\times \mathbb{Z}_2},\hspace{.2cm} {\mathbb{Z}_2\times \mathbb{Z}_3},\hspace{.2cm}  {\mathbb{Z}_2\times \mathbb{Z}_3},\hspace{.2cm} {\mathbb{Z}_2\times \mathbb{F}_4}.\]
\end{theorem}

 \begin{proof} Let $R$ be a non-local ring. Then $R \cong R_1 \times R_2 \times \cdots \times R_n$, where each $R_i$ is a local ring with maximal ideal $\frak{m_i}$.
 Suppose that $L(AG(R))$ is outer-planar. Let $n = 3$. Then the subgraph induced by the set $S$=$\{(1,0,0)$, $(0,1,0)$, $(1,0,1)$, $(1,1,0)$, $(0,0,1)\}$ is homeomorphic to  $K_1\vee P_3$. Thus, $L(AG(R))$ is not outer-planar for $n = 3$. Consequently, $L(AG(R))$ is not outer-planar for $n \geq 3$. It implies that $n=2$. 

 We may now suppose that $R\cong R_1\times R_2$. Assume that one of $R_i$ is not a field. Without loss of generality, assume $R_1$ is not a field. Then $|\frak{m_1}|\geq2$ and $|U(R_1)|\geq2$. Since $|\frak{m_1}|\geq2$, we can choose $a,b\in \frak{m_1}$ such that $ab=0$. Now, consider the set $T=\{(a,1),(b,0),(1,0),(u,0),(0,1)\}$, where $u\in U(R_1)\setminus\{1\}$. Note that the subgraph induced by the set $T$ is isomorphic to $K_{2,3}$. Therefore, $L(AG(R))$ is not outer-planar. Consequently, both $R_i$ are fields. 

 We now suppose that both  $R_1$ and $R_2$ are a field. Suppose $|R_i|>4$, for some $i$. Without loss of generality, assume that  $|R_1|>4$. Note that subgraph  induced by the set $ W=\{(1,0), (u_1,0),(u_2,0),(u_3,0),(0,1)\}$, where $u_i\in U(R_2)\setminus\{1\}$, is isomorphic to $K_{1,4}$ and so $L(AG(R))$ is not a outer-planar. Thus, $|R_i|\leq4$ for each $i$. Therefore, $R$ is isomorphic to one of the rings; ${\mathbb{Z}_2\times \mathbb{Z}_2},\hspace{.2cm} {\mathbb{Z}_2\times \mathbb{Z}_3},\hspace{.2cm}  {\mathbb{Z}_2\times \mathbb{F}_4},\hspace{.2cm}  {\mathbb{Z}_3\times \mathbb{F}_4},\hspace{.2cm}  {\mathbb{F}_4\times \mathbb{F}_4}$.

 If $R\cong \mathbb{F}_4\times \mathbb{F}_4$ or $\mathbb{Z}_3\times \mathbb{F}_4$, then by Remark \ref{fielddproduct}, $AG(R)$ contains $K_{2,3}$ as a subgraph. Thus, the graph $L(AG(R))$ is not outer-planar. Conversely, if  $R\cong \mathbb{Z}_3\times \mathbb{Z}_3$, then  $L(AG(R))=C_4$. If $R\cong \mathbb{Z}_2\times \mathbb{F}_4$, then  $L(AG(R))=K_3$. If $R\cong \mathbb{Z}_2\times \mathbb{Z}_3$, then  $L(AG(R))=P_2$. If $R\cong \mathbb{Z}_2\times \mathbb{Z}_2$, then  $L(AG(R))=K_1$. Thus, the graph $L(AG(R))$ is outer-planar, for the rings ${R\cong\mathbb{Z}_2\times \mathbb{Z}_2},\hspace{.2cm} {\mathbb{Z}_2\times \mathbb{Z}_3},\hspace{.2cm}  {\mathbb{Z}_3\times \mathbb{Z}_3},\hspace{.2CM} {\mathbb{Z}_2\times \mathbb{F}_4}$.
 \end{proof}

\begin{theorem}\label{localringsouterplnr} For the local ring $R$, the line graph $L(AG(R))$ is outer-planar if and only if the annihilator graph $AG(R)$ is outer-planar.
\begin{proof} By the similar argument, used in the proof of Theorem \ref{localringsplnr}, we get the result.
\end{proof}

\end{theorem}
In the view of Theorem \ref{localringsouterplnr} and {\cite[Theorem 2.3]{afkhami2017some}}, we have the following corollary:

\begin{corollary}
    
\label{outerplanar_lineannhilator(local)} For the finite local ring $(R,\frak{m})$, the line graph $L(AG(R))$ is outer planar if and only if $R$ is isomorphic to one of the following rings:
\[\mathbb{Z}_4, \hspace{.2cm}{\frac {\mathbb {Z}_{2}[x]}{(x^2)}},\hspace{.2cm} \mathbb{Z}_9,\hspace{.2cm}\frac{\mathbb {Z}_{3}[x]} {(x^3)},\hspace{.2cm}\mathbb{Z}_8,\hspace{.2cm} {\frac {\mathbb {Z}_{2}[x]}{(x^2)}},  \frac{\mathbb  {Z}_4[x]}{(2x,x^2-2)},\hspace{.2cm} \frac{ \mathbb {F}_{4}[x]}{(x^2)},\hspace{.2cm} \frac{\mathbb {Z}_{4}[x]}{(x^2+x+1)},\hspace{.2cm}  \frac{\mathbb {Z}_{4}[x]}{(x,2)},\hspace{.2cm}  \frac{\mathbb {Z}_{2}{[x,y]}}{(x,y)^2}. \]
    \end{corollary}

An embedding $\phi$ of a planar graph is labeled 1-outerplanar if it demonstrates outerplanarity, indicating that all vertices are linked to the outer face in the embedding. This notion is extended to designate an embedding as $k$-outerplanar if, upon removing all vertices on the outer face along with their incident edges, the result is a $(k - 1)$-outerplanar embedding. A graph is classified as $k$-outerplanar if it allows for such a $k$-outerplanar embedding. The \textit{outerplanarity index} of a graph $G$ is defined as the smallest value of $k$ for which $G$ is $k$-outerplanar. 
Regarding a planar graph $G$, the \textit{inner vertex number} $i(G)$ represents the smallest count of vertices not included in the boundary of the exterior region in any plane embedding of $G$. A graph $G$ is classified as minimally non-outerplanar when $i(G) = 1$. 
For a more in-depth exploration of $k$-outerplanarity, one may refer to \cite{frank,kulli}.

\begin{theorem}\label{outerplanarity-2}
  Let $R$ be a finite commutative ring. Then $L(AG(R))$ has outerplanarity index 2 if and only if $R$ is isomorphic to one of the following rings:
  \begin{center}
    $\mathbb{Z}_{25}$, $\frac{\mathbb{Z}_5[x]}{\left\langle x^2\right\rangle}$, $\mathbb{Z}_2 \times \mathbb{Z}_5$, $\mathbb{Z}_3 \times \mathbb{F}_4$, $\mathbb{Z}_2 \times \mathbb{Z}_4$, $\mathbb{Z}_2 \times \frac{\mathbb{Z}_2[x]}{(x^2)}$, $\mathbb{Z}_2 \times \mathbb{Z}_2 \times \mathbb{Z}_2$.
  \end{center}
\end{theorem}
\begin{proof}
 We know that every graph with outerplanarity index - 2 is also a planar graph. Thus, we have to check whether $L(AG(R))$ for the rings given in Theorem \ref{Planar_line annihilator(non-local)} and Corollary \ref{planar_lineannhilator(local)} has outerplanarity index  2. 
  
If $R \cong {\mathbb{Z}_2\times \mathbb{Z}_2},\hspace{.2cm} {\mathbb{Z}_2\times \mathbb{Z}_3},\hspace{.2cm}  {\mathbb{Z}_3\times \mathbb{Z}_3},\hspace{.2cm} {\mathbb{Z}_2\times \mathbb{F}_4}, \mathbb{Z}_4, \hspace{.2cm}{\frac {\mathbb {Z}_{2}[x]}{(x^2)}},\hspace{.2cm} \mathbb{Z}_9,\hspace{.2cm}\frac{\mathbb {Z}_{3}[x]} {(x^3)},\hspace{.2cm}\mathbb{Z}_8,\hspace{.2cm} {\frac {\mathbb {Z}_{2}[x]}{(x^2)}},  \frac{\mathbb  {Z}_4[x]}{(2x,x^2-2)},\hspace{.2cm} \frac{ \mathbb {F}_{4}[x]}{(x^2)},\hspace{.2cm} \frac{\mathbb {Z}_{4}[x]}{(x^2+x+1)},\hspace{.2cm}  \frac{\mathbb {Z}_{4}[x]}{(x,2)},\hspace{.2cm}$ or $\frac{\mathbb {Z}_{2}{[x,y]}}{(x,y)^2}$, then the outerplanarity index of $L(AG(R))$ is 1 by Theorem \ref{outerPlanar_line annihilator(non-local)} and Corollary \ref{outerplanar_lineannhilator(local)}. If $R \cong \mathbb{Z}_{25}$ or $\frac{\mathbb{Z}_5[x]}{\left\langle x^2\right\rangle}$, then $L(AG(R))$ is given in Figure 3(b). If we delete the vertices from the outer faces of the drawing, then the resultant graph is $K_3$, which is 1-outerplanar. Therefore, the outerplanarity index of $L(AG(R))$ is 2.
 If $R \cong \mathbb{Z}_2 \times \mathbb{Z}_5$, then $L(AG(R)) \cong K_4$. If we delete the vertices from the outer faces of the drawing, then the resultant graph is $K_1$, which is outerplanar. It implies that  $L(AG(R))$ is 2-outerplanar. If $R \cong  \mathbb{Z}_3 \times \mathbb{F}_4$, $\mathbb{Z}_2 \times \mathbb{Z}_4$ or $\mathbb{Z}_2 \times \frac{\mathbb{Z}_2[x]}{(x^2)}$, then $L(AG(R))$ is given in Figure 1. If we delete the vertices from the outer faces of the graph, then the resultant graph is $K_2$, which is outerplanar. Thus, $L(AG(R))$ is again 2-outerplanar. Finally, if $R \cong \mathbb{Z}_2 \times \mathbb{Z}_2 \times \mathbb{Z}_2$, then $L(AG(R))$ is given in Figure 2(b). Again, if we delete the vertices from the outer faces, then the resultant graph is outerplanar. Hence, the outerplanarity index of $L(AG(R))$ is 2.
\end{proof}

\begin{corollary}
  Let $R$ be a finite commutative ring. Then $L(AG(R))$ has an outerplanarity index at most two.
\end{corollary}

\begin{theorem}
  Let $R$ be a finite commutative ring. Then the inner vertex number of $L(AG(R))$ is given by:
  \begin{center}
    $i(L(AG(R)))=
\begin{cases}
3 & \text{if} ~ R \cong \mathbb{Z}_2 \times \mathbb{Z}_2 \times \mathbb{Z}_2;\\
2 & \text{if}~ R \cong \mathbb{Z}_3 \times \mathbb{F}_4, \mathbb{Z}_2 \times \mathbb{Z}_4 ~\text{or}~ \mathbb{Z}_2 \times \frac{\mathbb{Z}_2[x]}{(x^2)};\\
1 & \text{if}~ R \cong \mathbb{Z}_2 \times \mathbb{Z}_5;\\
0 & \text{otherwise}.
\end{cases}$
  \end{center}
\end{theorem}
\begin{proof}
  The proof follows from Theorem \ref{outerplanarity-2}, Figure 1 and Figure 2(b).
\end{proof}

\begin{corollary}
  Let $R$ be a finite commutative ring. Then $L(AG(R))$ is minimally non-outerplanar if and only if $R \cong \mathbb{Z}_2 \times \mathbb{Z}_5$.
\end{corollary}

\subsection{Genus of $\text{L(AG}(R))$}
In this subsection, we classify all the finite commutative rings with unity $R$ such that ${L(AG(R))}$ has a genus at most two.

First, we show some examples of rings whose line graphs are of small genera and crosscaps. 
\begin{remark} \cite[Example 2.12, 2.14]{MR2735063}\label{genusk_{3,3}} We have
   $ \gamma(L(K_{3,3})) = 1$, $\overline{\gamma}(L(K_{3,3})) = 1$, $\overline{\gamma}(L(K_{2,4}))= 2$, $\gamma(L(K_{3,4)}) = 2$, $\overline{\gamma}(L(K_{3,4)}) \geq3$ and
$\gamma(L(K_{4,4)})\geq 3$.
\end{remark}

\begin{example} \label{genuk_{2,5}}
$\gamma(L(K_{2,5}))= 2$, $\overline{\gamma}(L(K_{2,5})) = 2 $, $ \gamma(L(K_{3,5}))\geq 3$ and $\overline{\gamma}(L(K_{3,5})) \geq 3$.
 \begin{proof} Since the degre of two vertices in $K_{2,5}$ is 5 and so $\gamma(L(AG(R))\geq \gamma (K_5)+\gamma (K_5)\geq2$. Moreover, in Figure 4, we show the explicit embedding of  $L(K_{2,5})$ into a surface of genus two. By the Lemma \ref{genusK_{2,n}}, $\overline{\gamma}(L(K_{2,5})) = 2$.
Let $G = L(K_{3,5})$, $G_1 = L(K_{2,5})$, and $G_2 = K_5$. The graphs $G$, $G_1$, and $G_2$ satisfy the conditions of Lemma \ref{genusoftwograph}. Therefore, we have $\gamma(L(K_{3,5})) > \gamma(L(K_{2,5})) = 2$, which implies that $\gamma(L(K_{3,5})) \geq 3$.
  Similarly, $\overline{\gamma}(L(K_{3,5}))>\overline{\gamma}(L(K_{2,5}))=2$ and so $\overline{\gamma}(L(K_{3,5}))\geq 3$.
   \begin{figure}[h!]
\centering
\includegraphics[width=0.3 \textwidth]{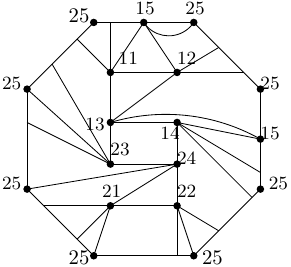}
\caption{$L(K_{2,5})$}
\label{line k}
\end{figure}  
 \end{proof}   
\end{example}
Let $G$ be a graph, and let $L(G)$ be its line graph. Suppose $G$ is a 2-cell embedded on an orientable (non-orientable) surface of genus $\gamma(G)$ (crosscap $\overline{\gamma}(G)$). Corresponding to each face $\mathcal{F}_i$ in $G$, there is a face $\mathcal{F}_i^*$ in $L(G)$, and this correspondence preserves the face length. Additionally, corresponding to the degree of each vertex $v$ in $G$, there is a face of length equal to $\deg(v)$ in $L(G)$ (see \cite{MR485482}).

\begin{example} \label{givenexample3}
Suppose $R=\mathbb{Z}_3\times \mathbb{Z}_4 $. Then $\gamma(L(AG(R))\ge3$ and $\overline{\gamma}(L(AG(R))\ge3$.
 \end{example}
\begin{proof}
    Since $V(AG(R)$ consists the vertices $v_1=(0,1)$, $v_2=(0,2)$, $v_3=(0,3)$, $v_4=(1,0)$, $v_5=(2,0)$, $v_6=(1,2)$, $v_7=(2,2)$. Now consider the subgraph $H$ of $AG(R)$, $H=AG(R)-\{W_{67}\}$, where $W_{67}=[v_6,v_7]$. It is easy to observe that  $H$ is isomorphic to $K_{3,4}$. Thus, $\gamma(L(AG(R))\geq\gamma(L(H))=2$. Since $\gamma({H})=1$, then by Euler’s formula $v - e + f = 2 - 2g$, $H$ has 5 faces as $H$ has 12 edges and 7 vertices. Suppose that $\{\mathcal{F}_1,\mathcal{F}_2, \mathcal{F}_3, \mathcal{F}_4, \mathcal{F}_5\}$ is the set of faces of $H$ in ${S}_1$. Let $|\mathcal{F}_i|$ denote the face length. Then we have $|\mathcal{F}_1|+|\mathcal{F}_2|+|\mathcal{F}_3|+|\mathcal{F}_4|+|\mathcal{F}_5|=24$. From this equality, observe that two faces are of length six and three faces are of length four in the embedding of $H$  in ${S}_1$. Corresponding to these five faces in the embedding of $H$ in ${S}_1$, there are three faces of length four and two faces of length six in the embedding  $L(H)$ in ${S}_2$. Since the degrees of $v_1$, $v_2$ and $v_3$ are four in $H$. Therefore, corresponding to these vertices, there are three faces of length four in the embedding $L(H)$ in ${S}_2$. Therefore, at least two faces of $L(H)$ have a length six, and at least six faces of $L(H)$ have a length four. By Euler's formula $v - e + f = 2 - 2g$, $L(H)$ have 16 faces as $L(H)$ has 30 edges and 12 verities. Let $\{{\mathcal{F}_1}^*, {\mathcal{F}_2}^*,\ldots, {\mathcal{F}_{16}}^*\}$ is the set of faces of $L(H)$. Suppose $|{\mathcal{F}_1}^*|=|{\mathcal{F}_2}^*|=|{\mathcal{F}_3}^*|=|{\mathcal{F}_4}^*|=|{\mathcal{F}_5}^*|=|{\mathcal{F}_6}^*|=4$ and  $|{\mathcal{F}_7}^*|=|{\mathcal{F}_8}^*|=6$. Since $\sum_{i=1}^{16} |{\mathcal{F}_i}^*|= 60$. It implies that $\sum_{i=9}^{16} |{\mathcal{F}_i}^*|= 24$. It concludes that every other face $\{{\mathcal{F}_9}^*,\ldots, {\mathcal{F}_{16}}^*\}$ is of length three. Observe that in $L(AG(R))$, the vertex $W_{6,7}$ has degree six. If we wish to recover the embedding of $L(AG(R))$ from $L(H)$, the vertex $W_{67}$ should be inserted into the face of length six. But every face of length six in the embedding of $H$ contains either vertex $v_4$ or $v_5$ and so both six length faces of $L(H)$ in $S_2$ has at least two vertices from  the set $ \{W_{1,4},W_{2,4},W_{3,4},W_{1,5},W_{2,5},W_{3,5}\}$. Therefore, there is no face of length six, which contains  $W_{6,7}$. Hence, we conclude that $\gamma(L(AG(R))\ge3$. Also by Remark \ref{genusk_{3,3}}, $\overline{\gamma}(L(AG(R)))\geq \overline{\gamma}(L(H))\geq3$.

\end{proof} 
 
\begin{example}\label{grapgexm pl}
    The genus and crosscap of the line graph of graph $G^{'}$ is greater than equal to $2$.
\begin{figure}[h!]
\centering
\includegraphics[width=0.3 \textwidth]{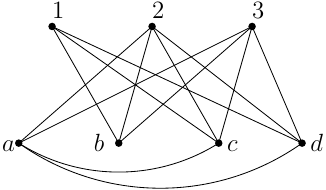}
\caption{$G^{'}$}
\label{line k}
\end{figure}
\begin{proof}

    Let $G=L(G^{'})$ and $W_{ij}=[x_i,x_j]$. Consider the graph $G_1=L(K_{3,3})$ and graph $G_2$ on the set with vertices $\{W_{2,a}, W_{3,a},W_{a,c},W_{a,d}\}$, which is complete graph $K_4$. The  graphs $G,G_1$ and $G_2$ satisfies the condition of Lemma \ref{genusoftwograph}, therefore we have $\gamma(L(G^{'}))>\gamma( L(K_{3,3}))\geq2$ and $\overline{\gamma}(L(AG(R))> \overline{\gamma}(L(K_{3,3}))\geq2$.  
\end{proof}
\end{example}
\begin{example} \label{givenexample2}
Suppose $R=\mathbb{Z}_2\times \mathbb{Z}_2 \times \mathbb{Z}_3$. Then $\gamma(L(AG(R))\ge3$ and $\overline{\gamma}(L(AG(R))\ge3$.
\begin{proof}
    Consider the vertices $x_1=(1,0,0)$, $x_2=(0,1,0)$, $x_3=(0,0,1)$, $x_4=(1,1,0)$, $x_5=(0,1,1)$, $x_6=(0,0,2)$, $x_7=(0,1,2)$, $x_8=(1,0,1)$, $x_9=(1,0,2)$. Let $W_{ij}=[x_i,x_j]$. Note that the graph $L(AG(R))$ contains a copy of graph $L(G^{'})$ with the vertices \{$W_{1,2}$,$W_{1,5}$,$W_{1,7}$,$W_{2,8}$,$W_{5,8}$,$W_{7,8}$,$W_{2,9}$,$W_{5,9}$,$W_{7,9}$,$W_{4,8}$,$W_{4,9}$,$W_{4,5}$,$W_{4,7}$\} and a copy of subgraph $H$ with vertices $\{W_{1,3},W_{2,3},W_{4,3},W_{1,6},W_{2,6},W_{4,6}\}$. Observe that $H$ is not an outer planar. Since each vertex of $H$ is joined to the vertices of $L(G^{'})$, we conclude that $\gamma(L(AG(R)\ge \gamma(L(G^{'}))+1\geq3$ by the Example \ref{grapgexm pl} and Lemma \ref{genusoftwograph}.
\end{proof}    
\end{example}

\begin{theorem}\label{genus1product}

Let $R$ be a non-local commutative ring. Then $\gamma({L(AG}(R)))= 1$ if and only if $R$ is isomorphic to one of the following four rings:
\[\mathbb {Z}_{3} \times \mathbb {Z}_{5}, \mathbb {F}_{4} \times \mathbb {F}_{4},  \mathbb {Z}_{2} \times \mathbb {Z}_{7}, \mathbb {Z}_{2} \times \mathbb {F}_{8}.\]
\end{theorem} 
\begin{proof} Let $R$ be a non-local ring. Then $R \cong R_1 \times R_2 \times \cdots \times R_n$, where each $R_i$ is a local ring with maximal ideal $\frak{m_i}$. Assume that $\gamma({L(AG}(R))= 1$. Let $n\geq4$. Now consider the vertices $x_1=(1,0,0,0,\ldots,0)$,\hspace{.1cm} $x_2=(0,1,0,0,\ldots,0)$,\hspace{.1cm} $x_3=(0,0,1,0,\ldots,0)$,\hspace{.1cm} $x_4=(0,0,0,1,\ldots,0)$,\hspace{.1cm} $x_5=(0,0,1,1,\ldots,0)$,\hspace{.1cm} $x_6=(0,1,0,1,\ldots,0)$,\hspace{.1cm} $x_7=(1,0,0,1,\ldots,0)$. Observe that $deg(x_1)\geq5$  and $deg(x_2)\geq5$. Thus,  by Theorem \ref{genusintermofdegree}, $\gamma({L(AG(R))})\geq \gamma(K_5)+\gamma(K_5)=2$,  which is not possible. Consequently, $n\leq 3$.
  
   We may now suppose that $R\cong R_1\times R_2\times R_3$. 
Let $|R_i|\geq3$ for some $i$. Without loss of generality, assume that $|R_3|\geq 3$. Consider the vertices $x_1=(1,0,0)$, $x_2=(0,1,0)$, $x_3=(0,0,1)$, $x_4=(0,0,v)$, $x_5=(1,0,1)$, $x_6=(1,0,v)$, $x_7=(0,1,1)$, $x_8=(0,1,v)$, $x_{10}=(1,1,0)$. Note that $deg(x_1)=5$ and $deg(x_2)=5$. Thus, 
${\gamma}(L(AG( {R}_{1} \times  {R}_{2}\times  {R}_{3})))\geq2$. It implies that $|R_i|\leq2$ and so  $R$ is isomorphic to $\mathbb{Z}_{2}\times \mathbb{Z}_{2}\times\mathbb {Z}_{2}$, which is not possible because ${L(AG}(\mathbb{Z}_{2}\times \mathbb{Z}_{2}\times\mathbb {Z}_{2}))$ is planar by Theorem \ref{Planar_line annihilator(non-local)}.

 Now, suppose that $R\cong R_1\times R_2$. Assume both $R_1$ and $R_2$ are not fields. Then we have $|\frak{m_i}|\geq2$ and $|U(R_i)|\geq 2$, for $i\in\{1,2\}$. Let $|\frak{m_i}|\geq3$ for some $i$. Without loss of generality, assume that  $|\frak{m_1}|\geq 3$. Then  $|U(R_1)|\geq4$. We can choose ${a,b}\in \frak{m_1}\setminus\{0\}$ such that $ab=0$. Consider the vertices $x_1=(a,1)$, $x_2=(b,1)$, $x_3=(0,1)$, $x_4=(1,0)$, $x_5=(u_1,0)$, $x_6=(u_2,0)$, $x_7=(u_3,0)$, where $\{u_1,u_2,u_3\}\in U(R_1)\setminus\{1\}$.  Note that  for $i\in\{4,5,6,7\}$,  $x_2\in ann(x_1x_i)$ and  
$x_2\not\in{ann(x_1)\bigcup ann(x_i)}$,   $x_1\in ann(x_2x_i)$ and  
$x_1\not\in{ann(x_2)\bigcup ann(x_i)}$. Then the subgraph induced by the set $S=\{x_1,x_2,x_3,x_4,x_5,x_6,x_7\}$ contains $K_{3,4}$ as a subgraph and so $\gamma(L(AG(R)))\geq \gamma(L(K_{3,4}))=2$. Consequently, $|\frak{m_i}|\leq2$ for both $i$. Now suppose $|\frak{m_1}|=2$ and $|R_2|\geq 3$. Consider the set $T=\{x_1,x_2,\dots,x_7\}$, where $x_1=(1,0)$, $x_2=(a,0)$, $x_3=(u,0)$, $x_4=(0,1)$, $x_5=(0,v)$, $x_6=(a,1)$, $x_7=(a,v)$, $v\in U(R_2)\setminus\mathbb\{1\}$. Note that for $6\leq i\leq 7$,
 $x_7\in ann(x_1x_i)$ but  $x_7\not\in ann(x_1)\cup ann(x_i)$,  $x_6\in ann(x_3x_i)$ but  $x_6\not\in ann(x_3)\cup ann(x_i)$. 
 Observe that the subgraph induced by the set $T$ is isomorphic to $K_{3,4}$. Therefore, $\gamma(L(AG(R))\geq \gamma(L(K_{3,4}))\geq 2$, leads to a contradiction. 
It implies that $|R_2|\leq 2$. Hence, it concludes that the ring $R$ is isomorphic to either $\mathbb{Z}_{4}\times\mathbb{Z}_{2}$ or $\frac{\mathbb{Z}_{2}[x]}{(x^2)}\times\mathbb{Z}_{2}$. This is also not possible by Theorem \ref{Planar_line annihilator(non-local)}.
Suppose that both  $R_1$ and $R_2$ are a field.
 Now, let $|R_i|\geq9$ for some $i$. Without loss of generality, assume that $|R_1|\geq9$. Then the subgraph induced by the set $\{(1,0), (0,1), (u_1,0), (u_2,0), (u_3,0),(u_4,0),(u_5,0),(u_6,0),(u_7,0),(u_8,0)\}$, where $u_i\in U(R_1)\setminus\{1\}$  is $K_{1,8}$ and so $ L(AG(R))$ contains $K_8$ as a subgraph. It follows that   ${\gamma} (L(AG(R)))\geq 2$. Thus, $|R_i|\leq 8$, for each $i$. Suppose $|R_i|\geq 5$ for both $i$. Then by Remark \ref{fielddproduct}, $AG(R)$ contains $K_{3,4}$ as a subgraph and so ${\gamma} (L(AG(R)))\geq \gamma (L(K_{3,4}))= 2$. Therefore, $R$ is isomorphic to one of the rings; $\mathbb {Z}_{2} \times \mathbb {Z}_{2}$, $\mathbb {Z}_{2} \times \mathbb {Z}_{3}$, $\mathbb {Z}_{2} \times \mathbb {F}_{4}$, $\mathbb {Z}_{2} \times \mathbb {Z}_{5}$, $\mathbb {Z}_{2} \times \mathbb {Z}_{7}$, $\mathbb {Z}_{2} \times \mathbb {F}_{8}$, $\mathbb {Z}_{3} \times \mathbb {Z}_{3}$, $\mathbb {Z}_{3} \times \mathbb {F}_{4}$, $\mathbb {Z}_{3} \times \mathbb {Z}_{5}$, $\mathbb {Z}_{3} \times \mathbb {Z}_{7}$, $\mathbb {Z}_{3} \times \mathbb {F}_{8}$, $\mathbb {F}_{4} \times \mathbb {F}_{4}$, $\mathbb {F}_{4} \times \mathbb {Z}_{5}$, $\mathbb {F}_{4} \times \mathbb {Z}_{7}$, $\mathbb {F}_{4} \times \mathbb {F}_{8}$.

 If $R$ is isomorphic to either $\mathbb {Z}_{3} \times \mathbb {Z}_{7}$ or $\mathbb {Z}_{3} \times \mathbb {F}_{8}$, then $AG(R)$ contains $K_{2,6}$ as a subgraph. Thus, by Lemma \ref{genusK_{2,n}},  $\gamma( L(AG(R)))\geq2$. If $R$ is isomorphic to one of the three rings:  $\mathbb {F}_{4} \times \mathbb {Z}_{5}$, $\mathbb {F}_{4} \times \mathbb {Z}_{7}$, $\mathbb {F}_{4} \times \mathbb {F}_{8}$, then $AG(R)$ contains $K_{3,4}$ as a subgraph.  Thus, by Example \ref{genusk_{3,3}}, $\gamma( L(AG(R)))\geq2$. If $R$ is isomorphic to one of the following rings:
 $\mathbb {Z}_{2} \times \mathbb {Z}_{2}$, $\mathbb {Z}_{2} \times \mathbb {Z}_{3}$, $\mathbb {Z}_{2} \times \mathbb {F}_{4}$, $\mathbb {Z}_{2} \times \mathbb {Z}_{5}$, $\mathbb {Z}_{3} \times \mathbb {Z}_{3}$, $\mathbb {Z}_{3} \times \mathbb {F}_{4}$, then by Theorem \ref{Planar_line annihilator(non-local)}, the graph $L(AG(R))$ is planar.

Conversely, if $R=\mathbb {Z}_{3} \times \mathbb {Z}_{5}$,  then $AG(R)$ is isomorphic to $K_{2,4}$.  Thus, by Remark \ref{genusk_{3,3}} ${\gamma}(L((AG (\mathbb {Z}_{3} \times \mathbb {Z}_{5})))= 1$. If $R=\mathbb {F}_{4} \times \mathbb {F}_{4}$,  then $AG(R)$ is isomorphic to  isomorphic to $K_{3,3}$.  Therefore,  by Remark \ref{genusk_{3,3}} ${\gamma}(L(AG (\mathbb {F}_{4} \times \mathbb {F}_{4})))= 1$. If $R\cong\mathbb {Z}_{2} \times \mathbb {Z}_{7}$ or $\mathbb {Z}_{2} \times \mathbb {F}_{8}$, then $L(AG(R))$ is  isomorphic to $K_6$ or $K_7$ and so  ${\gamma} (L(AG(R)))=1$.
\end{proof}

\begin{theorem}\label{genusofR1R2}
Let $R$ be a non-local commutative ring. Then $\gamma({L(AG(R))})= 2$ if and only if $R$ is isomorphic to one of the following rings: \[\mathbb {Z}_{2} \times \mathbb {F}_{9},\hspace{.2CM} \mathbb {Z}_{3} \times \mathbb {Z}_{7}, \hspace{.2CM} \mathbb {F}_{4} \times \mathbb {Z}_{5}.\]
  
\end{theorem}

\begin{proof} Let $R$ be a non-local ring. Then $R \cong R_1 \times R_2 \times \cdots \times R_n$, where each $R_i$ is a local ring with maximal ideal $\frak{m_i}$.
Assume that $\gamma({L(AG)}(R)))= 2$. Let $n\geq4$. Now consider the vertices $x_1=(1,0,0,0,\ldots,0)$,\hspace{.1cm} $x_2=(0,1,0,0,\ldots,0)$,\hspace{.1cm} $x_3=(0,0,1,0,\ldots,0)$,\hspace{.1cm} $x_4=(0,0,0,1,\ldots,0)$,\hspace{.1cm} $x_5=(1,1,0,0,\ldots,0)$,\hspace{.1cm} $x_6=(0,1,1,0,\ldots,0)$,\hspace{.1cm} $x_7=(0,0,1,1,\ldots,0)$,\hspace{.1cm} $x_8=(1,0,0,1,\ldots,0)$,\hspace{.1cm} $x_9=(1,0,1,0,\ldots,0)$,\hspace{.1cm} $x_{10}=(0,1,0,1,\ldots,0)$,\hspace{.1cm} $x_{11}=(0,1,1,1,\ldots,0)$,\hspace{.1cm} $x_{12}=(1,0,1,1,\ldots,0)$,\hspace{.1cm} $x_{13}=(1,1,0,1,\ldots,0)$,\hspace{.1cm} $x_{14}=(1,1,1,0,\ldots,0)$. Note that $deg(x_1)\geq7$, $deg(x_2)\geq7$ and $deg(x_3)\geq7$. Hence, by Corollary \ref{genusbydegree}, $\gamma({L(AG)}(R)))\geq3$. Consequently, $n\leq3$.

We may now suppose that $R=R_1\times R_2 \times R_3$. Suppose that one of $R_i$ is not a field. Without loss of generality, assume that $R_1$ is not a field. Then we have  $|\frak{m_1}|\geq2$ and $|U(R_1)|\geq2$.
 Let ${a,b}\in\frak{m_1}\setminus\{0\}  $ and $u\in U(R_1)\setminus\{1\} $ such that $ab=0$. Consider the vertices $x_1=(1,0,0)$, $x_2=(u,0,0)$, $x_3=(0,1,0)$, $x_4=(0,0,1)$, $x_5=(0,1,1)$, $x_6=(1,0,1)$, $x_7=(u,0,1)$, $x_8=(1,1,0)$, $x_9=(u,1,0)$, $x_{10}=(a,0,0)$, $x_{11}=(b,1,0)$, $x_{12}=(b,0,1)$, $x_{13}=(a,1,1)$. Observe that  ${L(AG}(R))$ contains a copy of $K_8$ with vertices $\{W_{1,13}, W_{2,13}, W_{6,13}, W_{7,13}, W_{8,13}, W_{9,13}, W_{11,13}, W_{12,13}\}$ and a copy of $K_4$ with the vertices $\{W_{1,2}, W_{2,3}, W_{3,6}, W_{3,7}\}$. Since each vertex of $K_4$ is adjacent to a vertex of $K_8$, by Lemma \ref{genusoftwograph}, we conclude that $\gamma({L(AG}(R))\geq \gamma(K_8)+1=3$, a contradiction. This implies that each $R_i$ is a field.
 Let $|R_i|\geq3$ for some $i$. Without loss of generality, assume that $|R_3|\geq 3$. Note that $AG( {R}_{1} \times  {R}_{2}\times {R}_{3})$ contains a subgraph isomorphic to $AG(\mathbb {Z}_{2} \times \mathbb {Z}_{2}\times \mathbb {Z}_{3})$. 
By Example \ref{givenexample2},  ${\gamma}(L(AG( {R}_{1} \times  {R}_{2}\times  {R}_{3})))\geq {\gamma}(L(AG(\mathbb {Z}_{2} \times \mathbb {Z}_{2}\times \mathbb {Z}_{3}))\geq3$. It implies that $|R_i|\leq2$ and so  $R$ is isomorphic to $\mathbb{Z}_{2}\times \mathbb{Z}_{2}\times\mathbb {Z}_{2}$, which is not possible because ${L(AG}(\mathbb{Z}_{2}\times \mathbb{Z}_{2}\times\mathbb {Z}_{2}))$ is planar by Theorem \ref{Planar_line annihilator(non-local)}. It follows that $n=2$.

Now suppose that $R=R_1\times R_2$. If both $R_1$ and $R_2$ are not fields, then we have $|\frak{m_i}|\geq2$ and $|U(R_i)|\geq2$ for $i=1,2$. We claim that if  ${|\frak{m_i}|}\geq 3$, for some $i$, then $\gamma(L(AG(R)))\geq3$. Suppose that one of $|\frak{m_i}|\geq3$. Without loss of generality, assume that ${|\frak{m_1}|}\geq 3$. First, suppose $|{\frak{m_1}}|= 3$, then only possible rings are ${\mathbb{Z}_{9}}$ or $\frac{\mathbb{Z}_{3}[x]}{(x^2)}$ and so $|U(R_1)|=6$. Let ${a,b}\in m_1\setminus \{0\}$ and $\{u_1,u_2,\ldots,u_5\}\in U(R_1)\setminus \{1\}$. Now consider the set $S=\{x_1,x_2,\ldots,x_9\}$, where $x_1=(a,1)$, $x_2=(b,1)$, $x_3=(0,1)$, $x_4=(1,0)$, $x_5=(u_1,0)$, $x_6=(u_2,0)$, $x_7=(u_3,0)$, $x_8=(u_4,0)$, and $x_9=(u_5,0)$. Note that, for $k\in\{4,5,6,7,8,9\}$, $x_2\in ann(x_1x_k)$, but $x_2\not\in {ann(x_1)\bigcup ann(x_k)}$, and $x_1\in ann(x_2x_k)$, but $x_1\not\in {ann(x_2)\bigcup ann(x_k)}$. Observe that the graph induced by the set $S$ contains $K_{3,6}$ as a subgraph. This implies that $\gamma(L(AG(R)))\geq \gamma(L(K_{3,6}))=3$, a contradiction.
Now, suppose $|\frak{m_1}|\geq4$, then $|U(R_1)|\geq4$. Let ${d,e.f}\in {\frak{m_1}}\setminus \{0\}$ and $\{v_1,v_2,v_3,\}\in U(R_1)\setminus \{1\}$. Consider $x_1=(d,1)$, $x_2=(e,1)$, $x_3=(f,1)$, $x_4=(0,1)$, $x_5=(v_1,0)$, $x_6=(v_2,0)$, $x_7=(v_3,0)$, and $x_8=(v_4,0)$. Then the graph induced by the set $\{x_1,x_2,x_3,x_4,x_5,x_6,x_7,x_8\}$ contains $K_{4,4}$ as a subgraph. Thus, $\gamma(L(AG(R)))\geq \gamma(L(K_{4,4}))\geq3$, which is not possible. Thus, $|{\frak{m_1}}|<3$. Consequently, $|\frak{m_i}|<3$ for each $i$.
% Without loss of generality, assume that $|m_1|\geq3$, then $|U(R_1)|\geq4$. Let $a,b\in m_1\setminus\{0\}$ such that $ab=0$, $\{u_1,u_2,u_3\}\in {U(R_1)\setminus \{1\}}$ and $v\in {U(R_2)\setminus \{1\}}$. Consider the vertices  $x_1=(a,1)$, $x_2=(b,1)$,  $x_3=(a,v)$, $x_4=(b,v)$, $x_5=(0,1)$, $x_6=(0,v)$, $y_1=(a,0)$, $y_2=(b,0)$ $y_3=(1,0)$, $y_4=(u_1,0)$, $y_5=(u_2,0)$, $y_6=(u_3,0)$, and $x_7=(u_4,0)$. Note that  for $i=3,4,5,6$,  $x_2\in ann(x_1y_i)$ but  
% $x_2\not\in{aan(x_1)\bigcup ann(y_i)}$,   $x_1\in ann(x_2y_i)$ but 
% $x_1\not\in{aan(x_2)\bigcup ann(y_i)}$,  $x_3y_i=0$. Thus, graph $AG(S)$ induced by the set $S=\{x_1,x_2,x_5,x_6,y_1,y_2,y_3,y_4\}$ is isomorphic to $K_{4,4}$ and so $g(L(AG(R)))\geq g(L(K_{4,4}))=3$
Now, suppose that $|\frak{m_i}|=2$ for both rings. Let $a\in \frak{m_1}\setminus\{0\}$, $b\in \frak{m_2}\setminus\{0\}$ and $u\in U(R_1)\setminus\{1\}$, $v\in U(R_2)\setminus\{1\}$. Note that $a^2=0=b^2$. Consider the vertices $x_1=(1,0)$, $x_2=(u,0)$, $x_3=(a,0)$, $x_4=(a,1)$, $x_5=(a,v)$, $x_6=(a,b)$, $x_7=(0,1)$, $x_8=(0,v)$, $x_9=(0,b)$, $x_{10}=(1,b)$, $x_{11}=(u,b)$. Note that $x_4\in ann(x_ix_6)$ but $x_4\notin{ann(x_i)\cup ann(x_6)}$ $(i=1,2)$, $x_{10}\in ann(x_jx_6)$ but $x_{10}\notin{ann(x_j)\cup ann(x_6)}$  $(j=4,5)$,  $x_{10}\in ann(x_kx_6)$ but $x_{10}\notin{ann(x_k)\cup ann(x_6)}$  $(k=7,8)$, $x_{4}\in ann(x_lx_6)$ but $x_{4}\notin{ann(x_l)\cup ann(x_6)}$  $(k=10,11)$ and $x_3x_6=0=x_6x_9$. This implies that $deg(x_6)=10$. Therefore, $\gamma(L(AG(R)))\geq \gamma(L(K_{10}))\geq3$. Consequently, one of $R_i$ must be a field.

%  Without loss of generality, assuming that $R_2$ is the field and $R_1$ is not the field, then we have ${|m_1|}\geq2 $. We claim that if  ${|m_1|}\geq 3$, then $g(L(AG(R)\geq3$. 
% Suppose $|{m_1}|= 3$, then only possible rings are ${\mathbb{Z}_{9}}$ or $\frac{\mathbb{Z}_{3}[x]}{(x^2)}$ and so $|U(R_1)|=6$. Let ${a,b}\in m_1$ such that $ab=0$ and $\{u_1,u_2,\ldots,u_5\}\in U(R_1)\setminus \{1\}$. Now consider a set of vertices $\{x_1,x_2,\ldots,x_9\}$, where $x_1=(a,1)$, $x_2=(b,1)$, $x_3=(0,1)$, $y_1=(1,0)$, $y_2=(u_1,0)$, $y_3=(u_2,0)$, $y_4=(u_3,0)$, $y_5=(u_4,0)$, and $y_6=(u_5,0)$. Clearly $x_3y_i=0$, for $i=1,2,3,4,5,6$. Note that $x_2\in ann(x_1y_j)$, but $x_2\not\in {ann(x_1)\bigcup ann(y_j)}(j=1,2,3,4,5,6)$, and $x_1\in ann(x_2y_k)$, but $x_1\not\in {ann(x_2)\bigcup ann(y_k)}(k=1,2,3,4,5,6)$. Observe that the graph induced by set $\{x_1,x_2,\ldots,x_9\}$ contain $K_{3,6}$ as a subgraph. This implies that $g(L(AG(R)\geq g(L(K_{3,6}))\geq3$, a contradiction.
% Suppose $|m_1|\geq4$, then $|U(R_1)|\geq4$. Let ${d,e.f}\in {m_1}$ such that $de=df=0$ and ${v_1,v_2,v_3,v_4}\in U(R_1)$. Consider $y_1=(d,1)$, $y_2=(e,1)$, $y_3=(f,1)$, $y_4=(0,1)$, $x_1=(v_1,0)$, $x_2=(v_2,0)$, $x_3=(v_3,0)$, and $x_4=(v_4,0)$. Then the graph induced by the set $\{y_1,y_2,y_3,y_4,x_1,x_2,x_3,x_4\}$ contains $K_{4,4}$ as a subgraph. Hence,  $g(L(AG(R))\geq g(L(K_{4,4}))\geq3$, not possible. Hence, $|{m_1}|<3$
Assume that $R_1$ is not a field and $R_2$ is a field.
Suppose $|\frak{m_1}|=2$ and $|R_2|\geq4$. Let $a\in {\frak{m_1}}\setminus\{0\}$ such that $a^2=0$, $u\in U(R_1)$ and ${v_1,v_2}\in U(R_2)\setminus\{1\} $. Now consider $x_1=(1,0)$, $x_2=(a,0)$, $x_3=(u,0)$, $y_1=(0,1)$, $y_2=(0,v_1)$, $y_3=(0,v_2)$, $y_4=(a,1)$, $y_5=(a,v_1)$, $y_6=(a,v_2)$. Note that $y_4\in ann(x_1y_j),y_4\not\in{ann(x_1)\bigcup ann(y_j)}$, where $j\in \{4,5,6\}$. 
%Suppose $|{m_1}*|=1$ and $|R_2|\geq4$. Let $a\in m_1$ such that $a^2=0$ and ${v_1,v_2}\in U(R_2) $. Now consider  $x_1=(1,0)$, $x_2=(a,0)$, $x_3=(u,0)$, $y_1=(0,1)$, $y_2=(0,v_1)$, $y_3=(0,v_2)$, $y_4=(a,1)$, $y_5=(a,v_1)$, $y_6=(a,v_2)$, where $u\in U(R_1)$. Clearly $x_1y_i=x_3y_i=0$ for $i=1,2,3$, $x_2y_i=0$, for $i=1,2,3,4,5,6$. Then $x_2\in ann(x_1y_j),x_2\not\in{ann(x_1)\bigcup ann(y_j)}$, where ($j=4,5,6$),  $x_2\in ann(x_3y_k),x_2\not\in{ann(x_3)\bigcup ann(y_k)}$, where ($k=4,5,6$). 
The subgraph induced by the  set $\{x_1,x_2,x_3,y_1,\ldots,y_6\}$ contains $K_{3,6}$ as a subgraph. Therefore, $\gamma(L(AG(R)))\geq \gamma(L(K_{3,6}))\geq 3$, a contradiction. This implies that $|R_2|\leq 3$. 
If $|R_2|=2$, then  $R$ is isomorphic to either ${\mathbb{Z}_4\times \mathbb{Z}_2}$ or ${\frac{\mathbb Z_2[x]}{(x^2)}\times {\mathbb {Z}_2}}$. If $|R_2|=3$, then $R$ is isomorphic to either ${\mathbb{Z}_4\times \mathbb{Z}_3}$ or ${\frac{\mathbb Z_2[x]}{(x^2)}\times {\mathbb {Z}_3}}$. If $R$ is isomorphic to either ${\mathbb{Z}_4\times \mathbb{Z}_2}$ or ${\frac{\mathbb Z_2[x]}{(x^2)}\times {\mathbb {Z}_2}}$, then by Theorem \ref{Planar_line annihilator(non-local)}, the graph $\gamma(L(AG(R)))$ is planar. If $R$ is isomorphic to either ${\mathbb{Z}_4\times \mathbb{Z}_3}$ or ${\frac{\mathbb Z_2[x]}{(x^2)}\times {\mathbb {Z}_3}}$, then by Example \ref{givenexample3}, $\gamma(L(AG(R)))\geq 3$.

Suppose that both  $R_1$ and $R_2$ are fields.
 Now, let $|R_i|>10$ for some $i$. Without loss of generality, assume that $|R_1|>10$. Then the subgraph induced by the set \{ $(1,0)$, $(0,1)$, $(u_1,0)$, $(u_2,0)$, $(u_3,0)$, $(u_4,0)$,$(u_5,0)$, $(u_6,0)$,$(u_7,0)$, $(u_8,0)$, $(u_9,0)$ ,$(u_{10},0)$\}, where $u_i\in U(R_1)\setminus\{1\}$, is $K_{1,10}$ and so $ L(AG(R))$ contains $K_{10}$ as a subgraph. It follows that   ${\gamma} (L(AG(R)))\geq 3$. Thus, $|R_i|\leq 9$, for each $i$. Suppose $|R_i|\geq 5$ for both $i$. Then by Remark \ref{fielddproduct}, the graph $AG(R)$ contains $K_{4,4}$ as a subgraph and so by Remark \ref{genusk_{3,3}}, ${\gamma} (L(AG(R)))\geq \gamma (L(K_{4,4}))\geq 3$.  Therefore, $R$ is isomorphic to one of the  rings; $\mathbb {Z}_{2} \times \mathbb {Z}_{2}$, $\mathbb {Z}_{2} \times \mathbb {Z}_{3}$, $\mathbb {Z}_{2} \times \mathbb {F}_{4}$, $\mathbb {Z}_{2} \times \mathbb {Z}_{5}$, $\mathbb {Z}_{2} \times \mathbb {Z}_{7}$, $\mathbb {Z}_{2} \times \mathbb {F}_{8}$, $\mathbb {Z}_{2} \times \mathbb {F}_{9}$, $\mathbb {Z}_{3} \times \mathbb {Z}_{3}$, $\mathbb {Z}_{3} \times \mathbb {F}_{4}$, $\mathbb {Z}_{3} \times \mathbb {Z}_{5}$, $\mathbb {Z}_{3} \times \mathbb {Z}_{7}$, $\mathbb {Z}_{3} \times \mathbb {F}_{8}$, $\mathbb {Z}_{3} \times \mathbb {F}_{9}$, $\mathbb {F}_{4} \times \mathbb {F}_{4}$, $\mathbb {F}_{4} \times \mathbb {Z}_{5}$, $\mathbb {F}_{4} \times \mathbb {Z}_{7}$, $\mathbb {F}_{4} \times \mathbb {F}_{8}$
, $\mathbb {F}_{4} \times \mathbb {F}_{9}$.

If $R$ is isomorphic to either  $\mathbb {Z}_{3} \times \mathbb {F}_{8}$ or $\mathbb {Z}_{3} \times \mathbb {F}_{9}$, then $AG(R)$ contains $K_{2,7}$ as a subgraph. Therefore, by Lemma \ref{genusK_{2,n}}, $\gamma(L(AG(R)))\geq 3$. If $R$ is isomorphic one of the rings:  $\mathbb {F}_{4} \times \mathbb {Z}_{7}$, $\mathbb {F}_{4} \times \mathbb {F}_{8}$, $\mathbb {F}_{4} \times \mathbb {F}_{9}$, then $AG(R)$ contains $K_{3,5}$ as a subgraph. Thus, by Example \ref{genuk_{2,5}}, $\gamma(L(AG(R)))\geq 3$. If $R$ is isomorphic to one of  five rings:
$\mathbb {Z}_{3} \times \mathbb {Z}_{5}, \mathbb {F}_{4} \times \mathbb {F}_{4},  \mathbb {Z}_{2} \times \mathbb {Z}_{7}, \mathbb {Z}_{2} \times \mathbb {F}_{8}$, then by Theorem \ref{genus1product}, $\gamma(L(AG(R)))= 1$. If $R$ is isomorphic to one of the following rings: $\mathbb {Z}_{2} \times \mathbb {Z}_{2}$, $\mathbb {Z}_{2} \times \mathbb {Z}_{3}$, $\mathbb {Z}_{2} \times \mathbb {F}_{4}$, $\mathbb {Z}_{2} \times \mathbb {Z}_{5}$, $\mathbb {Z}_{3} \times \mathbb {Z}_{3}$, $\mathbb {Z}_{3} \times \mathbb {F}_{4}$, then by Theorem \ref{Planar_line annihilator(non-local)}, the graph $L(AG(R))$ is planar. 

Conversely, if $R\cong\mathbb {Z}_{2} \times \mathbb {F}_{9}$, then $AG(R)$ is isomorphic to $K_{1,8}$ and so $L(AG(R))$ is isomorphic to $K_8$. Thus, $\gamma(L(AG(R)))=2$. If $R\cong\mathbb {Z}_{3} \times \mathbb {Z}_{7}$, then $AG(R)$ is isomorphic to $K_{2,6}$ and so by Lemma \ref{genusK_{2,n}}, $\gamma(L(AG(R)))=2$. If $R\cong\mathbb {F}_{4} \times \mathbb {Z}_{5}$, then $AG(R)$ is isomorphic to $K_{3,4}$ and so by Remark \ref{genusk_{3,3}}, $\gamma(L(AG(R)))=2$.
\end{proof}

\begin{theorem}\label{genus-localring} There does not exist any finite local ring $(R,\frak{m})$, for which $\gamma(L(AG(R))\leq2$.
\begin{proof}
    Since $R$ is a finite local ring. We have $AG(R)=K_r$, where $r=|m^*|$. If $1\leq r\leq4$, then $L(AG(R))$ is planar. If $r\geq 6$, then by Lemma \ref{linecompltegraph}, $\gamma(L(AG(R))\geq3$. For $r=5$, we have $|m|=6$, which is not possible.
\end{proof}
    
\end{theorem}

\subsection{Crosscap of $\text{L(AG}(R))$}

In this subsection, we classify all the commutative rings $R$ with unity such that the crosscap of ${L(AG}(R))$ is at most two.

\begin{theorem}\label{crosscapone}
Let $R$ be a non-local commutative ring. Then $ \overline{\gamma} (L(AG(R)))=1$ if and only if $R$ is isomorphic to one of the following rings: $\mathbb {Z}_{2} \times \mathbb {Z}_{7}$,\hspace{.2cm}$\mathbb {F}_{4} \times \mathbb {F}_{4}$.
\end{theorem}
\begin{proof}
    Let $R$ be a non-local ring. Then $R \cong R_1 \times R_2 \times \cdots \times R_n$, where each $R_i$ is a local ring with maximal ideal $\frak{m_i}$.  Suppose $ \overline{\gamma} ({L(AG}(R)))=1$.
    Let $n=4$. Consider the vertices $x_1=(1,0,0,0)$,  $x_2=(0,1,0,0)$,  $x_3=(0,0,1,0)$,  $x_4=(0,0,0,1)$,  $x_5=(0,0,1,1)$,  $x_6=(0,1,1,0)$,  $x_7=(1,0,0,1)$. Note that $deg(x_1)=5$ and   $deg(x_2)=5$. It follows that  $ \overline{\gamma} (L(AG(R)))\geq \overline{\gamma}(K_5)+ \overline{\gamma}(K_5)\geq2$. Consequently, $ \overline{\gamma} (L(AG(R)))\geq 2$ for $n\geq4$. Thus, $n\leq3$.
   
 We may now suppose that $R\cong R_1\times R_2\times R_3$. 
     % First, assume that one of $R_i$ is not a field. Without loss of generality, suppose that $R_1$ is not a field. Then, $|\frak{m_1}|\geq2$ and $|U(R_1)|\geq2$. We can choose $a,b\in {\frak{m_1}}$ such that $ab=0$. Consider the vertices $x_1=(1,0,0)$, $x_2=(a,0,0)$, $x_3=(1,0,1)$, $x_4=(u,0,1)$, $x_5=(b,1,1)$, $x_6=(b,1,0)$, $x_7=(a,1,0)$, $x_8=(a,1,1)$, where $u\in U(R_1)\setminus\{1\}$.
    % Note that $x_7\in ann(x_ix_5)$ but $x_7\not\in (ann(x_i)\cup ann(x_5)$, for $i\in\{1,3,5\}$, $x_8\in ann(x_jx_6)$ but $x_8\not\in (ann(x_j)\cup ann(x_6)$, for $i\in\{1,3,5\}$. Observe that the subgraph induced by the set $T=(x_1,x_2,x_3,x_4,x_5,x_6)$ contains $K_{2,4}$ and so $\overline{\gamma} (L(AG(R)))\geq \overline{\gamma}(K_{2,4})\geq2$.Thus, each $R_i$ is a field.
      Assume that $|R_i|\geq3$ for some $i$. Without loss of generality, let $|R_1|\geq 3$. Consider the set $T=\{$$(1,0,0)$, $(u,0,0)$, $(1,1,0)$, $(u,1,0)$, $(0,0,1)$, $(0,1,1)\}$, where $u\in U(R_1)\setminus\{1\}$. The subgraph induced by the set $T$ is isomorphic to $K_{2,4}$. Thus, $\overline{\gamma} (L(AG(R)))\geq \overline{\gamma}(L(K_{2,4}))\geq2$.  Consequently, $|R_i|\leq2$. Therefore, the ring $R$ is isomorphic to $ \mathbb {Z}_{2} \times \mathbb {Z}_{2} \times \mathbb {Z}_{2}$.  By Theorem \ref{Planar_line annihilator(non-local)}, the graph $L(AG( \mathbb {Z}_{2} \times \mathbb {Z}_{2} \times \mathbb {Z}_{2})$ is planar. 
      
    Now, we suppose that  $R\cong R_1\times R_2$. Assume that for some $i$, $R_i$ is not a field. Let $|\frak{m_i}|\geq3$ for some $i$.
     % If both $R_1$ and $R_2$ are not fields, then $|\frak{m_i}| \geq 2$ and $|U(R_i)| \geq 2$ for all $i$. Therefore,  we  can choose $a,b\in \frak{m_1}$, $c,d\in \frak{m_2}$ such that $ab=0$ and $cd=0$. Consider the vertices  $x_1=(1,0)$, $x_2=(u,0)$, $x_3=(a,0)$, $x_4=(b,1)$, $x_5=(b,0)$, $x_6=(0,1)$, $x_7=(0,v)$, $x_8=(0,c)$, $x_9=(0,d)$, $x_{10}=(1,d)$, $x_{11}=(b,d)$, where $u\in U(R_1)\setminus\{1\}$ and $v\in U(R_2)\setminus\{1\}$. Note that $deg(x_3)\geq 5$ and $deg(x_8)\geq 5$. It implies that $ \overline{\gamma} (L(AG(R)))\geq \overline{\gamma}(K_5)+ \overline{\gamma}(K_5)\geq2$. Thus, one of $R_i$ is field. Without loss of generality, assume that $R_2$ is a field, while $R_1$ is not a field.
     Without loss of generality, assume that $|\frak{m_1}|\geq3$. Then $|U(R_1)| \geq 4$. We can choose $a,b\in \frak{m_1}\setminus\{0\}$ such that $ab=0$. Consider the vertices $x_1=(a,0)$, $x_2=(b,0)$, $x_3=(a,1)$, $x_4=(b,1)$, $x_5=(1,0)$, $x_6=(u_1,0)$, $x_7=(u_2,0)$, where $u_1,u_2\in U(R_1)\setminus\{1\}$.
    Note that $x_5\in ann(x_3x_4)$ but $x_5\not\in (ann(x_3)\cup ann(x_4))$, $x_4\in ann(x_3x_i)$ but $x_4\not\in (ann(x_3)\cup ann(x_i))$, $x_3\in ann(x_ix_4)$ but $x_3\not\in (ann(x_i)\cup ann(x_4))$, for $i\in\{5,6,7\}$. Therefore,  $deg(x_3)\geq 5$ and $deg(x_4)\geq 5$. 
    It follows that $ \overline{\gamma} (L(AG(R)))\geq \overline{\gamma}(K_5)+ \overline{\gamma}(K_5)\geq2$. Thus, $|\frak{m_i}|\leq2$ for each $i$. 
    % so only possible choice for ring $R_1$ are $\mathbb{Z}_4$ and $\frac{\mathbb{Z}_4[x]}{(x^2)}$.
    Now, suppose that $|\frak{m_1}|=2$ and $|R_2|\geq3$. Then the subgraph induced by the set $\{(1,0), (u,0),(0,1),(0,v),(a,1),(a,v)\}$, where  $a\in \frak{m_1}\setminus\{0\}$, $u\in U(R_1)\setminus\{1\}$ and $v\in U(R_2)\setminus\{1\}$ contains $K_{2,4}$ as a subgraph. Thus, $\overline{\gamma} (L(AG(R)))\geq 2$. It follows that $|\frak{m_1}|=2$ and $|R_2|=2$. Therefore, the ring $R$ is isomorphic to either  $\mathbb {Z}_{4} \times \mathbb {Z}_{2}$ or $\frac{\mathbb{Z}_2[x]}{(x^2)} \times \mathbb {Z}_{2}$.  If $R$ is isomorphic to either ${\mathbb{Z}_4\times \mathbb{Z}_2}$ or ${\frac{\mathbb Z_2[x]}{(x^2)}\times {\mathbb {Z}_2}}$, then by Theorem \ref{Planar_line annihilator(non-local)}, the graph $\gamma(L(AG(R)))$ is planar. Consequently, both $R_1$ and $R_2$ are fields.
  Now suppose that both $R_1$ and $R_2$ are fields. Let $|R_i|>7$ for some $i$. Without loss of generality, assume that $|R_1|>7$. Then the subgraph induced by the set $\{(1,0), (0,1), (u_1,0), (u_2,0), (u_3,0),(u_4,0),(u_5,0),(u_6,0)\}$, where $u_i\in U(R_1)\setminus\{1\}$,  is $K_{1,7}$ and so $ L(AG(R))$ contains $K_7$ as a subgraph.  It follows that $\overline{\gamma} (L(AG(R)))\geq 3$. Thus, $|R_i|\leq 7$, for each $i$. Suppose $|R_i|\geq 5$ for both $i$. Then by Remark \ref{fielddproduct}, $AG(R)$ contains $K_{2,4}$ as a subgraph. It implies that $\overline{\gamma} (L(AG(R)))\geq 2$. Therefore, the only possible rings are $\mathbb {Z}_{2} \times \mathbb {Z}_{2}$, $\mathbb {Z}_{2} \times \mathbb {Z}_{3}$, $\mathbb {Z}_{2} \times \mathbb {F}_{4}$, $\mathbb {Z}_{2} \times \mathbb {Z}_{5}$, $\mathbb {Z}_{2} \times \mathbb {Z}_{7}$, $\mathbb {Z}_{3} \times \mathbb {Z}_{3}$, $\mathbb {Z}_{3} \times \mathbb {F}_{4}$, $\mathbb {Z}_{3} \times \mathbb {Z}_{5}$, $\mathbb {Z}_{3} \times \mathbb {Z}_{7}$, $\mathbb {F}_{4} \times \mathbb {F}_{4}$, $\mathbb {F}_{4} \times \mathbb {Z}_{5}$, $\mathbb {F}_{4} \times \mathbb {Z}_{7}$. If $R$ is one of the four rings: $\mathbb {Z}_{3} \times \mathbb {Z}_{5}$, $\mathbb {Z}_{3} \times \mathbb {Z}_{7}$, $\mathbb {F}_{4} \times \mathbb {Z}_{5}$, $\mathbb {F}_{4} \times \mathbb {Z}_{7}$, then $AG(R)$ contains $K_{2,4}$ as a subgraph. Therefore, $\overline{\gamma} (AG (R)\geq 2$. If  $R$ is isomorphic to one of the following rings: $\mathbb {Z}_{2} \times \mathbb {Z}_{2}$, $\mathbb {Z}_{2} \times \mathbb {Z}_{3}$, $\mathbb {Z}_{2} \times \mathbb {F}_{4}$, $\mathbb {Z}_{2} \times \mathbb {Z}_{5}$, $\mathbb {Z}_{3} \times \mathbb {Z}_{3}$, $\mathbb {Z}_{3} \times \mathbb {F}_{4}$, then by Theorem \ref{Planar_line annihilator(non-local)}, the graph $L(AG(R))$ is planar

     Conversely, if $R=\mathbb {Z}_{2} \times \mathbb {Z}_{7}$, then $AG(R)= K_{1,6}$ and so $L(AG(R))$ is $K_6$. Thus, $\overline{\gamma} (L(AG(\mathbb {Z}_{2} \times \mathbb {Z}_{7})))= 1$. If $ R=\mathbb {F}_{4} \times \mathbb {F}_{4}$, then $AG(R)= K_{3,3}$. Therefore, by Remark \ref{genusk_{3,3}}, $\overline{\gamma} (L(AG(\mathbb {F}_{4} \times \mathbb {F}_{4})))= 1$.
\end{proof}

\begin{theorem}
Let $R$ be a non-local commutative ring. Then  $ \overline{\gamma} (L(AG(R)))=2$ if and only if $R\cong\mathbb {Z}_{3} \times \mathbb {Z}_{5}$.
\end{theorem}
\begin{proof}
    Let $R$ be a non-local ring. Then $R \cong R_1 \times R_2 \times \cdots \times R_n$, where each $R_i$ is a local ring with maximal ideal $\frak{m_i}$.
    Suppose $ \overline{\gamma} ({L(AG}(R)))=2$.
    Let $n=4$. Now, consider the vertices $x_1=( 1,0,0,0)$,\hspace{.1cm}$x_2=( 0,1,0,0)$,\hspace{.1cm}$x_3=( 0,0,1,0)$,\hspace{.1cm}$x_4=( 0,0,0,1)$,\hspace{.1cm}$x_5=( 1,1,0,0)$,\hspace{.1cm}$x_6=( 0,1,1,0)$,\hspace{.1cm}$x_7=( 0,0,1,1)$,\hspace{.1cm}$x_8=( 1,0,0,1)$,\hspace{.1cm}$x_9=( 1,0,1,0)$,\hspace{.1cm}$x_{10}=( 0,1,0,1)$,\hspace{.1cm}$x_{11}=( 1,1,0,1)$. Note that $deg(x_1)\geq 5$,  $deg(x_2)\geq 5$ and $deg(x_3)\geq 7$. Thus, by the Corollary \ref{genusbydegree},  $\overline{\gamma}({L(AG}(R)))\geq 3$. Consequently, for $n\geq4$, we have $\overline{\gamma}({L(AG)}(R)))\geq 3$. therefore, $n\leq 3 $.

 We may now suppose that $R\cong R_1\times R_2\times R_3$.
  First, assume that one of $R_i$ is not a field. Without loss of generality, assume that $R_1$ is not a field. Then, $|\frak{m_1}|\geq2$ and $|U(R_1)|\geq2$.  Consider the vertices $x_1=( 1,0,0)$,\hspace{.1cm} $x_2=( u,0,0)$,\hspace{.1cm} $x_3=( a,0,0)$,\hspace{.1cm}  $x_4=( a,0,1)$,\hspace{.1cm}  $x_5=( a,1,0)$,\hspace{.1cm}  $x_6=( 0,1,1)$,\hspace{.1cm} $x_7=( 0,0,1)$, where $a\in \frak{m_1}\setminus\{0\}$ and $u\in U(R_1)\setminus\{1\}$. Note that the subgraph induced by the set $\{x_1,x_2,\ldots x_7\}$ contains $K_{3,4}$ as a subgraph. Therefore,  $\overline{\gamma} ({L(AG)}(R)))\geq \overline{\gamma}( L(K_{3,4}))\geq 3$, a contradiction. Thus, each $R_i$ is a field. Now, assume $R_i$ is a field for each $i$.
Let $|R_i|\geq3$ for some $i$. Without loss of generality, assume that $|R_3|\geq 3$. Note that $AG( {R}_{1} \times  {R}_{2}\times {R}_{3})$ contains a subgraph isomorphic to $AG(\mathbb {Z}_{2} \times \mathbb {Z}_{2}\times \mathbb {Z}_{3})$. 
By Example \ref{givenexample2},  $\overline{\gamma}(L(AG( {R}_{1} \times  {R}_{2}\times  {R}_{3})))\geq \overline{\gamma}(L(AG(\mathbb {Z}_{2} \times \mathbb {Z}_{2}\times \mathbb {Z}_{3}))\geq3$. It implies that $|R_i|\leq2$ and so  $R$ is isomorphic to $\mathbb{Z}_{2}\times \mathbb{Z}_{2}\times\mathbb {Z}_{2}$, which is not possible because ${L(AG}(\mathbb{Z}_{2}\times \mathbb{Z}_{2}\times\mathbb {Z}_{2}))$ is planar by Theorem \ref{Planar_line annihilator(non-local)}.

 Now, suppose that $R\cong R_1\times R_2$. Assume both $R_1$ and $R_2$ are not fields. Then we have $|\frak{m_i}|\geq2$ and $|U(R_i)|\geq 2$, for $i\in\{1,2\}$. Let $|\frak{m_i}|\geq3$ for some $i$. Without loss of generality, assume that  $|\frak{m_1}|\geq 3$. Then  $|U(R_1)|\geq4$. We can choose ${a,b}\in \frak{m_1}\setminus\{0\}$ such that $ab=0$. Now consider the vertices $x_1=(a,1)$, $x_2=(b,1)$, $x_3=(0,1)$, $x_4=(0,v)$, $x_5=(1,0)$, $x_6=(u_1,0)$, $x_7=(u_2,0)$, $x_8=(u_3,0)$, where $\{u_1,u_2,u_3\}\in U(R_1)\setminus\{1\}$ and $v\in U(R_2)\setminus\{1\}$. Define $W_{ij}=[x_i,x_j]$.  Note that the subgraph of $L(AG(R))$ induced by the set $S={\{W_{3,5}, W_{3,6}, W_{3,7}, W_{3,8}, W_{4,5}, W_{4,6}, W_{4,7}, W_{4,8}\}}$ is isomorphic to $L(K_{2,4})$ and the subgraph $H$ induced by the set $T={\{ W_{1,5}, W_{1,6}, W_{1,7}, W_{2,5}, W_{2,6}, W_{2,7}\}}$ contains a subgraph homeomorphic to $K_4$. Consequently, $H$ is not outer-planar. Since each vertex of $H$ is adjacent to at least one vertex of the subgraph induced by the set $S$. Hence, $ \overline{\gamma}({L(AG}(R)))\geq \overline{\gamma}(L(K_{2,4}))+1\geq 3$, which is not  possible. It implies that $|\frak{m_i}|\leq 2$ for each $i$.
 % Now, suppose that $|\frak{m_i}|=2$ for both rings. Let $a\in \frak{m_1}$, $b\in \frak{m_2}$ and $u\in U(R_1)\setminus\{1\}$, $v\in U(R_2)\setminus\{1\}$. Then consider the vertices $x_1=(1,0)$, $x_2=(u,0)$, $x_3=(a,0)$, $x_4=(a,1)$, $x_5=(a,v)$, $x_6=(a,b)$, $x_7=(0,1)$, $x_8=(0,v)$, $x_9=(0,b)$, $x_{10}=(1,b)$, $x_{11}=(u,b)$. One can observe that $deg(x_6)\geq7$, $deg(x_4)\geq5$ and $deg(x_{10})\geq7$. Consequently, $ \overline{\gamma}({L(AG}(R)))\geq3$ by Corollary \ref{genusbydegree}.
Now, let $|\frak{m_1}|=2$ and $|R_2|\geq 3$. Consider a set $S=\{x_1,x_2,x_3,x_4,x_5,x_6, x_7\}$, where $x_1=(1,0)$, $x_2=(a,0)$, $x_3=(u,0)$, $x_4=(0,1)$, $x_5=(0,v)$, $x_6=(a,1)$, $x_7=(a,v)$, where $a\in \frak{m_1}\setminus\{0\}$, ${u}\in U(R_1)\setminus\{1\}$ and ${v}\in U(R_2)\setminus\{1\}$. Then the graph induced by the set $S$ contains $K_{3,4}$ as a subgraph. Consequently, $\overline{\gamma} ({L(AG)}(R)))\geq 3$, leading to a contradiction. 
This implies that $|R_2|\leq 2$. Hence, it concludes that the ring $R$ is isomorphic to either $\mathbb{Z}_{4}\times\mathbb{Z}_{2}$ or $\frac{\mathbb{Z}_{4}[x]}{(x^2)}\times\mathbb{Z}_{2}$. This is also not possible by Theorem \ref{Planar_line annihilator(non-local)}.
%Suppose $|R_2|\geq 3$. consider a set $S=\{x_1,x_2,x_3,x_4,x_5, x_6, x_7\}$, where $x_1=(1,0)$,$x_2=(a,0)$,$x_3=(u,0)$,$x_4=(0,1)$,$x_5=(0,v)$,$x_6=(a,1)$,$x_7=(a,v)$, where $a\in m_1$, ${1,u}\in U(R_1)$, ${1,v}\in U(R_2)$. Then graph induced by the set $S$ contains $k_{3,4}$. Hence $ \Gamma (\text{L(AG)}(R)))\geq 3$, again a contradiction. 

 Suppose that both  $R_1$ and $R_2$ are a field.
 Now, let $|R_i|>7$ for some $i$. Without loss of generality, assume that $|R_1|>7$. Then the subgraph induced by the set $\{(1,0), (0,1), (u_1,0), (u_2,0), (u_3,0),(u_4,0),(u_5,0),(u_6,0)\}$, where $u_i\in U(R_1)\setminus\{1\}$  is isomorphic to $K_{1,7}$ and so $ L(AG(R))$ contains $K_7$ as a subgraph.  It follows that   $\overline{\gamma} (L(AG(R)))\geq 3$. Thus, $|R_i|\leq 7$, for each $i$. Suppose $|R_i|\geq 5$ for both $i$. Then by Remark \ref{fielddproduct}, $AG(R)$ contains $K_{3,4}$ as a subgraph. It implies that $\overline{\gamma} (L(AG(R)))\geq 3$. Therefore, the only possible rings are $\mathbb {Z}_{2} \times \mathbb {Z}_{2}$, $\mathbb {Z}_{2} \times \mathbb {Z}_{3}$, $\mathbb {Z}_{2} \times \mathbb {F}_{4}$, $\mathbb {Z}_{2} \times \mathbb {Z}_{5}$, $\mathbb {Z}_{2} \times \mathbb {Z}_{7}$, $\mathbb {Z}_{3} \times \mathbb {Z}_{3}$, $\mathbb {Z}_{3} \times \mathbb {F}_{4}$, $\mathbb {Z}_{3} \times \mathbb {Z}_{5}$, $\mathbb {Z}_{3} \times \mathbb {Z}_{7}$, $\mathbb {F}_{4} \times \mathbb {F}_{4}$, $\mathbb {F}_{4} \times \mathbb {Z}_{5}$, $\mathbb {F}_{4} \times \mathbb {Z}_{7}$.
 If $R =\mathbb {Z}_{3} \times \mathbb {Z}_{7}$, then $AG(R)$ is $K_{2,6}$. By Lemma \ref{genusK_{2,n}}, $\overline{\gamma}(AG(\mathbb {Z}_{3} \times \mathbb {Z}_{7}))\geq 4$.
 If $R$ is either $\mathbb {F}_{4} \times \mathbb {Z}_{5}$ or $\mathbb {F}_{4} \times \mathbb {Z}_{7}$, then $AG(R)$ contains $K_{3,4}$. Therefore, $\overline{\gamma} (AG (R)\geq 3$. 
If $R$ is isomorphic to  either $\mathbb {Z}_{2} \times \mathbb {Z}_{7}$ or $\mathbb {F}_{4} \times \mathbb {F}_{4}$, then by Theorem \ref{crosscapone}, $\overline{\gamma} (L(AG(R)))= 1$. If $R$ is isomorphic to one of the following  rings: $\mathbb {Z}_{2} \times \mathbb {Z}_{2}$, $\mathbb {Z}_{2} \times \mathbb {Z}_{3}$, $\mathbb {Z}_{2} \times \mathbb {F}_{4}$, $\mathbb {Z}_{2} \times \mathbb {Z}_{5}$, $\mathbb {Z}_{3} \times \mathbb {Z}_{3}$, $\mathbb {Z}_{3} \times \mathbb {F}_{4}$, then  by Theorem \ref{Planar_line annihilator(non-local)}, $L(AG(R))$ is planar.
    
 Conversely, if $R=\mathbb {Z}_{3} \times \mathbb {Z}_{5}$, then $AG(R)$ is isomorphic to $K_{2,4}$.  Thus, by Remark \ref{genusk_{3,3}}, $\overline{\gamma} (AG (\mathbb {Z}_{3} \times \mathbb {Z}_{5}))= 2$.
    \end{proof}   

\begin{theorem}\label{croscap-localring} There does not exist any finite local ring $(R,\frak{m})$, for which $\overline{\gamma}(L(AG(R))\leq2$.
\begin{proof} By the similar argument, used in the proof of Theorem \ref{genus-localring}, we get the result.
    % Since $R$ is a finite local ring. We have $(AG(R)=K_r$, where $r=|m^*|$. If $1\geq r\geq4$, then $L(AG(R)$ is planar. If $6\geq$, then $\gamma(L(AG(R))\geq3$ by Lemma \ref{linecompltegraph}. For $r=5$, we have $|m|=6$, which is not possible.
\end{proof}
    
\end{theorem}

\subsection{Book thickness of $\text{L(AG}(R))$}
In this section, we determine the book thickness of the graph $L(AG(R))$ whose genus is at most one. First of all, we find out the book thickness of planar $L(AG(R))$ arising from rings listed in Theorems \ref{Planar_line annihilator(non-local)} and \ref{planar_lineannhilator(local)} and we prove that all planar $L(AG(R))$ have book thickness at most two.

\par An $n$-\textit{book embedding} consists of a set of $n$ half-planes called pages, whose boundaries are bound together along a single line called the spine. If the vertices of a graph can be embedded on the spine of a book, and the edges can be placed in $r$ pages such that every edge lies on exactly one page, and no two edges cross within a given page, then the embedding is called an \textit{$r$-book embedding}. The book thickness of a graph $G$, denoted by $bt(G)$, is the smallest integer $n$ for which $G$ has an $n$-book embedding. 
For details on the notion of embedding of graphs in a surface and book embedding, one can refer to \cite{book,white1985graphs}.

The following results given by \cite{book} help us to prove the main results of this section.
\begin{theorem}\label{tb1}\cite[Theorem 2.5]{book}
 Let $G$ be a connected graph. Then the following holds:
\begin{itemize}
    \item[(i)] $bt(G) = 0$ if and only if $G$ is a path;
    \item[(ii)] $bt(G) < 1$ if and only if $G$ is outer planar.
\end{itemize}

\end{theorem}
\begin{theorem}\label{tb2}\cite[Theorems 3.4, 3.5, 3.6]{book} For the complete graph $K_p$ and complete bipartite graph $K_{p,q}$, we have 
\begin{itemize}
    \item[(i)]  $bt(K_p) =\left\lceil \frac{p}{2} \right\rceil$, where $p\geq 4$.
  \item[(ii)]  $bt(K_{p,q}) = p$, where $p \leq q$ with $q \geq p^2-p+1$.
  \item[(iii)]  $bt(K_{3,3}) = 3$ and $bt(K_{p,p}) = p-1$, where $p \geq 4$.
\end{itemize}
\end{theorem}

\begin{theorem}
  Let $R$ be a finite commutative ring whose $L(AG(R))$ is planar. Then the following hold:
  \begin{itemize}

    \item[(i)] $bt(L(AG(R))) = 0$ if and only if $R \cong \mathbb{Z}_4, \hspace{.2cm}{\frac {\mathbb {Z}_{2}[x]}{(x^2)}},\hspace{.2cm} \mathbb{Z}_9,\hspace{.2cm}\frac{\mathbb {Z}_{3}[x]} {(x^3)},\hspace{.2cm} \mathbb{Z}_2\times \mathbb{Z}_2,\hspace{.2cm} {\mathbb{Z}_2\times \mathbb{Z}_3}$;
    \item[(ii)]  $bt(L(AG(R))) = 1$ if and only if $R \cong \mathbb{Z}_8,\hspace{.2cm} {\frac {\mathbb {Z}_{2}[x]}{(x^2)}},  \frac{\mathbb  {Z}_4[x]}{(2x,x^2-2)},\hspace{.2cm} \frac{ \mathbb {F}_{4}[x]}{(x^2)},\hspace{.2cm} \frac{\mathbb {Z}_{4}[x]}{(x^2+x+1)},\hspace{.2cm}  \frac{\mathbb {Z}_{4}[x]}{(x,2)},\hspace{.2cm}$, $\frac{\mathbb {Z}_{2}{[x,y]}}{(x,y)^2}, \hspace{.2cm}  {\mathbb{Z}_3\times \mathbb{Z}_3},\hspace{.2cm} {\mathbb{Z}_2\times \mathbb{F}_4}$;
   \item[(iii)]  $bt(L(AG(R))) = 2$ if and only if $R \cong \mathbb{Z}_{25}$, $\frac{\mathbb{Z}_5[x]}{\left\langle x^2\right\rangle}$, $\mathbb{Z}_2 \times \mathbb{Z}_5$, $\mathbb{Z}_3 \times \mathbb{F}_4$, $\mathbb{Z}_2 \times \mathbb{Z}_4$, $\mathbb{Z}_2 \times \frac{\mathbb{Z}_2[x]}{(x^2)}$, $\mathbb{Z}_2 \times \mathbb{Z}_2 \times \mathbb{Z}_2$.
   \end{itemize}\
\end{theorem}
\begin{proof}
  The proof of parts (i) and (ii) follows from Theorems \ref{outerPlanar_line annihilator(non-local)},  \ref{outerplanar_lineannhilator(local)} and \ref{tb1}.\\
 
 For the proof of part (iii), we have to discuss about the rings given in Theorem \ref{outerplanarity-2} for which $L(AG(R))$ is a planar graph but not outerplanar.
 \par If $R \cong \mathbb{Z}_{25}$ or $\frac{\mathbb{Z}_5[x]}{\left\langle x^2\right\rangle}$, then 2-book embedding of $L(AG(R))$ is given in Figure 6.
  If $R \cong \mathbb{Z}_2 \times \mathbb{Z}_5$, then $L(AG(R)) \cong K_4$. Thus, by Theorem \ref{tb2}, $bt(L(AG(R))) = 2$. If $R \cong \mathbb{Z}_3 \times \mathbb{F}_4$, $\mathbb{Z}_2 \times \mathbb{Z}_4$, $\mathbb{Z}_2 \times \frac{\mathbb{Z}_2[x]}{(x^2)}$, then 2-book embedding of $L(AG(R))$ is given in Figure 7.  Finally, if $R \cong \mathbb{Z}_2 \times \mathbb{Z}_2 \times \mathbb{Z}_2$, then 2-book embedding of $L(AG(\mathbb{Z}_2 \times \mathbb{Z}_2 \times \mathbb{Z}_2))$ is given in Figure 8.
\end{proof}
\begin{center}
\begin{tikzpicture}[scale=0.5,ultra thick]
%\draw[help lines] (0,0) grid (10,9);
\draw [line width=0.20mm,black](0,0) -- (14,0);
 \draw[thick] (2,0) arc[start angle=180, end angle=0, radius=1cm];
 \draw[thick] (2,0) arc[start angle=180, end angle=0, radius=2cm];
\draw[thick] (2,0) arc[start angle=180, end angle=0, radius=4cm];
\draw[thick] (2,0) arc[start angle=180, end angle=0, radius=5cm];

\draw[thick] (4,0) arc[start angle=180, end angle=0, radius=1cm];
\draw[thick] (4,0) arc[start angle=180, end angle=360, radius=2cm];
\draw[thick] (4,0) arc[start angle=180, end angle=360, radius=4cm];

\draw[thick] (6,0) arc[start angle=180, end angle=0, radius=1cm];
\draw[thick] (6,0) arc[start angle=180, end angle=0, radius=2cm];

\draw[thick] (8,0) arc[start angle=180, end angle=0, radius=1cm];
\draw[thick] (8,0) arc[start angle=180, end angle=360, radius=2cm];

\draw[thick] (10,0) arc[start angle=180, end angle=0, radius=1cm];
\filldraw (2,0) circle (1pt);
\filldraw (4,0) circle (1pt);
\filldraw (6,0) circle (1pt);
\filldraw (8,0) circle (1pt);
\filldraw (12,0) circle (1pt);
\filldraw (10,0) circle (1pt);

\node[] at (7,-5) {{Figure 6. 2-book embedding of $L(AG(\mathbb{Z}_{25})) \cong L(AG(\frac{\mathbb{Z}_5[x]}{\left\langle x^2\right\rangle}))$.}};

\node[scale = 0.5] at (1.5,-0.5) {13};

\node[scale = 0.5] at (3.5,-0.5) {14};

\node[scale = 0.5] at (5.5,-0.5) {12};
\node[scale = 0.5] at (7.5,-0.5) {24};
\node[scale = 0.5] at (9.5,-0.5) {23};
\node[scale = 0.5] at (11.5,-0.5) {34};
\end{tikzpicture}

%\caption{Whatever caption you want} \end{figure}

%\begin{figure}
%\centering
\end{center}

\begin{center}
\begin{tikzpicture}[scale=0.5,ultra thick]
%\draw[help lines] (0,0) grid (10,9);
\draw [line width=0.20mm,black](0,0) -- (14,0);
 \draw[thick] (2,0) arc[start angle=180, end angle=0, radius=1cm];
 \draw[thick] (2,0) arc[start angle=180, end angle=0, radius=2cm];
\draw[thick] (2,0) arc[start angle=180, end angle=0, radius=4cm];
%\draw[thick] (2,0) arc[start angle=180, end angle=0, radius=5cm];

\draw[thick] (4,0) arc[start angle=180, end angle=0, radius=1cm];
\draw[thick] (4,0) arc[start angle=180, end angle=360, radius=4cm];

\draw[thick] (6,0) arc[start angle=180, end angle=0, radius=1cm];

\draw[thick] (8,0) arc[start angle=180, end angle=0, radius=1cm];
\draw[thick] (8,0) arc[start angle=180, end angle=360, radius=2cm];

\draw[thick] (10,0) arc[start angle=180, end angle=0, radius=1cm];
\filldraw (2,0) circle (1pt);
\filldraw (4,0) circle (1pt);
\filldraw (6,0) circle (1pt);
\filldraw (8,0) circle (1pt);
\filldraw (12,0) circle (1pt);
\filldraw (10,0) circle (1pt);

\node[] at (7,-5) {{Figure 7. 2-book embedding of $L(AG(\mathbb{Z}_3 \times \mathbb{F}_4)) \cong L(AG(\mathbb{Z}_2 \times \mathbb{Z}_4)) \cong L(AG(\mathbb{Z}_2 \times \frac{\mathbb{Z}_2[x]}{(x^2)}))$}};

\node[scale = 0.5] at (1.5,-0.5) {11};

\node[scale = 0.5] at (3.5,-0.5) {12};

\node[scale = 0.5] at (5.5,-0.5) {13};
\node[scale = 0.5] at (7.5,-0.5) {23};
\node[scale = 0.5] at (9.5,-0.5) {21};
\node[scale = 0.5] at (11.5,-0.5) {22};
\end{tikzpicture}

%\caption{Whatever caption you want} \end{figure}

%\begin{figure}
%\centering
\end{center}

\begin{center}
\begin{tikzpicture}[scale=0.5,ultra thick]
%\draw[help lines] (0,0) grid (10,9);
\draw [line width=0.20mm,black](0,0) -- (20,0);
 \draw[thick] (2,0) arc[start angle=180, end angle=0, radius=1cm];
 \draw[thick] (2,0) arc[start angle=180, end angle=0, radius=2cm];
\draw[thick] (2,0) arc[start angle=180, end angle=0, radius=4cm];
\draw[thick] (2,0) arc[start angle=180, end angle=0, radius=8cm];

\draw[thick] (4,0) arc[start angle=180, end angle=0, radius=1cm];
\draw[thick] (4,0) arc[start angle=180, end angle=360, radius=6cm];
\draw[thick] (4,0) arc[start angle=180, end angle=360, radius=7cm];

\draw[thick] (6,0) arc[start angle=180, end angle=0, radius=1cm];
\draw[thick] (6,0) arc[start angle=180, end angle=360, radius=4cm];

\draw[thick] (8,0) arc[start angle=180, end angle=0, radius=1cm];
\draw[thick] (8,0) arc[start angle=180, end angle=360, radius=2cm];
\draw[thick] (8,0) arc[start angle=180, end angle=360, radius=3cm];

\draw[thick] (10,0) arc[start angle=180, end angle=0, radius=1cm];
\draw[thick] (10,0) arc[start angle=180, end angle=0, radius=4cm];

\draw[thick] (12,0) arc[start angle=180, end angle=0, radius=1cm];
\draw[thick] (12,0) arc[start angle=180, end angle=0, radius=2cm];

\draw[thick] (14,0) arc[start angle=180, end angle=0, radius=1cm];

\draw[thick] (16,0) arc[start angle=180, end angle=0, radius=1cm];
\filldraw (2,0) circle (1pt);
\filldraw (4,0) circle (1pt);
\filldraw (6,0) circle (1pt);
\filldraw (8,0) circle (1pt);
\filldraw (10,0) circle (1pt);
\filldraw (12,0) circle (1pt);
\filldraw (14,0) circle (1pt);
\filldraw (16,0) circle (1pt);
\filldraw (18,0) circle (1pt);

\node[] at (10,-8) {{Figure 8. 2-book embedding of $L(AG(\mathbb{Z}_2 \times \mathbb{Z}_2 \times \mathbb{Z}_2))$}};

\node[scale = 0.5] at (2,-0.5) {$e_1$};

\node[scale = 0.5] at (3.5,-0.5) {$e_3$};

\node[scale = 0.5] at (5.5,-0.5) {$e_8$};
\node[scale = 0.5] at (7.5,-0.5) {$e_6$};
\node[scale = 0.5] at (10,-0.5) {$e_7$};
\node[scale = 0.5] at (11.5,-0.5) {$e_4$};

\node[scale = 0.5] at (13.5,-0.5) {$e_5$};
\node[scale = 0.5] at (15.5,-0.5) {$e_9$};
\node[scale = 0.5] at (17.5,-0.5) {$e_2$};
\end{tikzpicture}

%\caption{Whatever caption you want} \end{figure}

%\begin{figure}
%\centering
\end{center}

\begin{center}
\begin{tikzpicture}[scale=0.5,ultra thick]
%\draw[help lines] (0,0) grid (10,9);
\draw [line width=0.20mm,black](0,0) -- (18,0);
 \draw[thick] (2,0) arc[start angle=180, end angle=0, radius=1cm];
 \draw[thick] (2,0) arc[start angle=180, end angle=0, radius=3cm];
\draw[thick] (2,0) arc[start angle=180, end angle=0, radius=7cm];
\draw[thick] (2,0) arc[start angle=180, end angle=360, radius=2cm];

\draw[thick] (4,0) arc[start angle=180, end angle=0, radius=1cm];
\draw[thick] (4,0) arc[start angle=180, end angle=0, radius=2cm];
\draw[thick,dotted] (4,0) arc[start angle=180, end angle=360, radius=5cm];

\draw[thick] (6,0) arc[start angle=180, end angle=0, radius=1cm];
\draw[thick] (6,0) arc[start angle=180, end angle=360, radius=3cm];

\draw[thick] (8,0) arc[start angle=180, end angle=0, radius=1cm];

\draw[thick] (10,0) arc[start angle=180, end angle=0, radius=1cm];
\draw[thick] (10,0) arc[start angle=180, end angle=0, radius=2cm];
\draw[thick] (10,0) arc[start angle=180, end angle=0, radius=3cm];

\draw[thick] (12,0) arc[start angle=180, end angle=0, radius=1cm];

\draw[thick] (12,0) arc[start angle=180, end angle=360, radius=2cm];
\draw[thick] (14,0) arc[start angle=180, end angle=0, radius=1cm];

\filldraw (2,0) circle (1pt);
\filldraw (4,0) circle (1pt);
\filldraw (6,0) circle (1pt);
\filldraw (8,0) circle (1pt);
\filldraw (10,0) circle (1pt);
\filldraw (12,0) circle (1pt);
\filldraw (14,0) circle (1pt);
\filldraw (16,0) circle (1pt);

\node[] at (9,-6) {{Figure 9. 3-book embedding of $L(AG(\mathbb{Z}_3 \times \mathbb{Z}_5))$}};

\node[scale = 0.5] at (1.5,-0.5) {$13$};

\node[scale = 0.5] at (3.5,-0.5) {$14$};

\node[scale = 0.5] at (5.5,-0.5) {$15$};
\node[scale = 0.5] at (7.5,-0.5) {$16$};
\node[scale = 0.5] at (10,-0.5) {$26$};
\node[scale = 0.5] at (11.5,-0.5) {$25$};

\node[scale = 0.5] at (13.5,-0.5) {$24$};
\node[scale = 0.5] at (15.5,-0.5) {$23$};

\end{tikzpicture}

%\caption{Whatever caption you want} \end{figure}

%\begin{figure}
%\centering
\end{center}

\begin{theorem}
  Let $R$ be a finite commutative ring whose $L(AG(R))$ is toroidal. Then the following hold:
  \begin{itemize}
      \item[(i)]  $bt(L(AG(R))) = 3$ if and only if $R \cong \mathbb{Z}_3 \times \mathbb{Z}_5, \mathbb{F}_4 \times \mathbb{F}_4, \mathbb{Z}_2 \times \mathbb{Z}_7$; 
   \item[(ii)]   $bt(L(AG(R))) = 4$ if and only if $R \cong \mathbb{Z}_2 \times \mathbb{F}_8$.
  \end{itemize}
\end{theorem}
\begin{proof}
  Since $L(AG(R))$ is toroidal, we have to discuss about the rings given in Theorem \ref{genusofR1R2}. If $R \cong \mathbb{Z}_3 \times \mathbb{Z}_5$, then $L(AG(R)) \cong L(K_{2,4})$. Thus, the 3-book embedding of $L(AG(R))$ is given in Figure 9. If $R \cong \mathbb{F}_4 \times \mathbb{F}_4$, then $L(AG(R)) \cong L(K_{3,3})$. Therefore, the 3-book embedding of $L(AG(R))$ is given in Figure 10. If $R \cong \mathbb{Z}_2 \times \mathbb{Z}_7$, then $L(AG(R)) \cong K_6$. By Theorem \ref{tb2}, $bt(L(AG(R))) = 3$. Finally, if $R \cong \mathbb{Z}_2 \times \mathbb{Z}_8$, then $L(AG(R)) \cong K_7$. Hence, again by Theorem \ref{tb2}, $bt(L(AG(R))) = 4$.
\end{proof}

\begin{center}
\begin{tikzpicture}[scale=0.5,ultra thick]
%\draw[help lines] (0,0) grid (10,9);
\draw [line width=0.20mm,black](0,0) -- (20,0);
 \draw[thick] (2,0) arc[start angle=180, end angle=0, radius=1cm];
 \draw[thick] (2,0) arc[start angle=180, end angle=0, radius=2cm];
\draw[thick] (2,0) arc[start angle=180, end angle=0, radius=5cm];
\draw[thick] (2,0) arc[start angle=180, end angle=0, radius=6cm];

\draw[thick] (4,0) arc[start angle=180, end angle=0, radius=1cm];
\draw[thick] (4,0) arc[start angle=180, end angle=360, radius=3cm];
\draw[thick] (4,0) arc[start angle=180, end angle=360, radius=6cm];

\draw[thick] (6,0) arc[start angle=180, end angle=0, radius=1cm];
\draw[thick,dotted] (6,0) arc[start angle=180, end angle=360, radius=6cm];

\draw[thick] (8,0) arc[start angle=180, end angle=0, radius=2cm];
\draw[thick] (8,0) arc[start angle=180, end angle=360, radius=1cm];
\draw[thick,dotted] (8,0) arc[start angle=180, end angle=360, radius=5cm];

\draw[thick] (10,0) arc[start angle=180, end angle=0, radius=1cm];
\draw[thick] (10,0) arc[start angle=180, end angle=360, radius=3cm];

\draw[thick] (12,0) arc[start angle=180, end angle=360, radius=1cm];

\draw[thick] (14,0) arc[start angle=180, end angle=0, radius=1cm];
\draw[thick] (16,0) arc[start angle=180, end angle=0, radius=1cm];
\draw[thick] (14,0) arc[start angle=180, end angle=0, radius=2cm];
\filldraw (2,0) circle (1pt);
\filldraw (4,0) circle (1pt);
\filldraw (6,0) circle (1pt);
\filldraw (8,0) circle (1pt);
\filldraw (10,0) circle (1pt);
\filldraw (12,0) circle (1pt);
\filldraw (14,0) circle (1pt);
\filldraw (16,0) circle (1pt);
\filldraw (18,0) circle (1pt);

\node[] at (10,-7) {{Figure 10. 3-book embedding of $L(AG(\mathbb{F}_4 \times \mathbb{F}_4))$}};

\node[scale = 0.5] at (1.5,-0.5) {$14$};

\node[scale = 0.5] at (3.5,-0.5) {$15$};

\node[scale = 0.5] at (5.5,-0.5) {$16$};
\node[scale = 0.5] at (7.5,-0.5) {$26$};
\node[scale = 0.5] at (9.5,-0.5) {$25$};
\node[scale = 0.5] at (11.5,-0.5) {$24$};

\node[scale = 0.5] at (13.5,-0.5) {$34$};
\node[scale = 0.5] at (15.5,-0.5) {$35$};
\node[scale = 0.5] at (17.5,-0.5) {$36$};

\end{tikzpicture}

%\caption{Whatever caption you want} \end{figure}

%\begin{figure}
%\centering
\end{center}
%%%%%%%%%%%%%%%%%%%%%%%%%%%%%%%%%%%%%%%%%%%%%%%%%%%%%%%%%%%%%%%%%%%%%%%%%%%%%%%%%%%%%%%%%%%%%%%%%%%%%%%%%%%%%%%%%%%%%%%%%%%%%%%%%%%%%%%%%%%%%%%%%%%%%%%%%%

\section*{Declarations}

\textbf{Acknowledgement:} The first and second authors gratefully acknowledge Birla Institute of Technology and Science (BITS) Pilani, India, for providing financial support.

\vspace{.3cm}
\textbf{Conflicts of interest/Competing interests}: There is no conflict of interest regarding the publishing of this paper. 

\vspace{.3cm}
\textbf{Availability of data and material (data transparency)}: Not applicable.

\vspace{.3cm}
\textbf{Code availability (software application or custom code)}: Not applicable.

 % \newpage
 % \bibliographystyle{abbrv}
	% \bibliography{ref}
 
 % \begin{thebibliography}{10}

\end{document}